\documentclass[11pt]{amsart}
\usepackage{amsfonts,amsmath,amsthm,amssymb}
\usepackage{hyperref}
\usepackage{graphicx}
\usepackage[tmargin=1in,lmargin=1in,rmargin=1in,bmargin=1in]{geometry}
\newcommand{\fr}[2]{\frac{#1}{#2}}
\newcommand{\tfr}[2]{\tfrac{#1}{#2}}

\newcommand{\mf}[1]{\mathbf{#1}}
\newcommand{\mr}[1]{\mathrm{#1}}
\newcommand{\mcal}[1]{\mathcal{#1}}
\newcommand{\opn}[1]{\operatorname{#1}}

\newcommand{\R}{{\mathbb{R}}}
\newcommand{\qtq}[1]{\quad\text{#1}\quad}
\newcommand{\eps}{\varepsilon}

\theoremstyle{definition}

\newtheorem*{rmk}{Remark}
\newtheorem*{rmks}{Remarks}
\theoremstyle{plain}

\newtheorem{thm}{Theorem}[section]
\newtheorem{prop}[thm]{Proposition}
\newtheorem{lma}[thm]{Lemma}
\newtheorem{cor}[thm]{Corollary}

\numberwithin{equation}{section}
\begin{document}

	\title{Mass-critical inverse Strichartz theorems for 1d Schr\"{o}dinger 
	operators}
\author{Casey Jao, Rowan Killip, and Monica Visan}
	\maketitle
	\label{chapter:m-crit_inv_str}
	
	\begin{abstract}
          We prove inverse Strichartz theorems at $L^2$ regularity for
          a family of Schr\"{o}dinger evolutions in one space
          dimension.  Prior results rely on spacetime Fourier analysis
          and are limited to the translation-invariant equation
          $i\partial_t u = -\tfr{1}{2} \Delta u$.  Motivated by
          applications to the mass-critical Schr\"odinger equation
          with external potentials (such as the harmonic oscillator),
          we use a physical space approach.
	\end{abstract}
	
\section{Introduction}

In this paper, we prove an inverse Strichartz theorem for 
certain 
Schr\"{o}dinger 
evolutions on the real line with $L^2$ initial data. 
Recall that solutions to the linear Schr\"{o}dinger equation
\begin{align}
\label{e:intro_free_particle}
i\partial_t u = -\tfrac{1}{2} \Delta u \quad\text{with}\quad u(0, \cdot) \in L^2(\mf{R}^d),
\end{align}
satisfy the Strichartz inequality
\begin{align}
\label{e:intro_str}
\| u\|_{L^{\fr{2(d+2)}{2}}_{t,x} (\mf{R} \times \mf{R}^d)} \le C \| u(0, 
\cdot) 
\|_{L^2 (\mf{R}^d)}.
\end{align}
In this translation-invariant setting, it was proved that if $u$ comes close to saturating the 
above inequality, then the initial data $u(0)$ must exhibit some ``concentration"; see \cite{carles-keraani,merle-vega,moyua-vargas-vega,begout-vargas}.  We seek 
analogues of this result  when the right side of~\eqref{e:intro_free_particle} is replaced by a more general Schr\"{o}dinger 
operator $-\tfr{1}{2}\Delta + V(t,x)$. 

Such refinements of the Strichartz inequality have provided a key technical tool in the study of the
$L^2$-critical nonlinear Schr\"{o}dinger equation
\begin{align}
\label{e:intro_m-crit}
i\partial_t u = -\tfrac{1}{2}\Delta u \pm |u|^{\fr{4}{d}} u \quad\text{with}\quad u(0, \cdot) \in 
L^2(\mf{R}^d).
\end{align}
The term ``$L^2$-critical" or ``mass-critical" refers to the property that the 
rescaling
\begin{align*}
u(t,x) \mapsto u_{\lambda}(t,x) := \lambda^{\fr{d}{2}} u(\lambda^2 t, \lambda x), 
\quad \lambda > 0,
\end{align*}
preserves both the class of solutions and the conserved \emph{mass}
$
M[u] := \| u(t)\|_{L^2}^2 = \| u(0)\|_{L^2}^2.
$
An inverse theorem for~\eqref{e:intro_str} 
begets profile decompositions that underpin the large data theory 
by revealing 
how potential blowup 
solutions may concentrate. The reader 
may consult for instance the notes~\cite{claynotes} for a more 
detailed account of this connection. The initial-value problem \eqref{e:intro_m-crit} was shown to be globally wellposed in \cite{dodson_mass-crit_high-d,dodson_mass-crit_2d,dodson_mass-crit_1d, dodson_focusing, Killip2009, Killip2008, Tao2007}.

Characterizing near-optimizers of the 
inequality~\eqref{e:intro_str} 
involves significant technical challenges due to the presence of
noncompact symmetries. Besides invariance under rescaling and translations in 
space and time, the inequality also possesses \emph{Galilean invariance}
\begin{align*}
u(t,x)\mapsto u_{\xi_0}(t, x) := e^{i[x \xi_0 - \fr{1}{2}t |\xi_0|^2]} u(t, x - t 
\xi_0), \quad u_{\xi_0}(0) = e^{ix \xi_0} u(0), \quad \xi_0 \in \mf{R}^d.
\end{align*}
Because of this last degeneracy, the $L^2$-critical setting 
is much more delicate compared to variants of~\eqref{e:intro_str} with higher 
regularity Sobolev norms on the right side, such as the energy-critical 
analogue
\begin{align}
\label{e:intro_e-crit_str}
\| u\|_{L^{\fr{2(d+2)}{d-2}}_{t,x}(\mf{R} \times \mf{R}^d)} \le C \| \nabla 
u(0, \cdot)\|_{L^2(\mf{R}^d)}.
\end{align}
In particular, Littlewood-Paley theory has
little use when seeking an inverse to \eqref{e:intro_str} because $u(0)$ can concentrate anywhere in 
frequency space, not necessarily near the origin. The works cited above use
spacetime orthogonality arguments and appeal to Fourier restriction theory, 
such as Tao's bilinear estimate for paraboloids (when $d \ge 3$)~\cite{Tao2003}.

Ultimately, we wish to consider the large data theory for the equation
\begin{align}
\label{e:intro_m-crit_v}
i\partial_t u =  -\fr{1}{2}\Delta  u + V  u \pm |u|^{\fr{4}{d}} u, 
\quad u(0, \cdot) \in L^2(\mf{R}^d),
\end{align}
where $V(x)$ is a real-valued potential.  The main example we have in mind is 
the harmonic oscillator $V = 
\sum_j \omega_j^2 
x_j^2$, which has obvious physical relevance and arises in the study of 
Bose-Einstein condensates~\cite{zhang_bec}. Although the scaling symmetry is 
broken, solutions initially concentrated at a point 
are well-approximated for short times by (possibly modulated) solutions to the 
genuinely scale-invariant mass-critical
equation~\eqref{e:intro_m-crit}. As described in Lemma~\ref{l:galilei} below, 
the harmonic oscillator also admits a more 
complicated analogue of Galilean invariance. This is related to the fact that 
$u$ solves equation~\eqref{e:intro_m-crit} iff its \emph{Lens transform} 
$\mcal{L} u$ satisfies equation~\eqref{e:intro_m-crit_v} with $V = 
\tfr{1}{2}|x|^2$, where
\begin{align*}
\mcal{L} u (t,x) := \tfr{1}{(\cos t)^{d/2}} u \Bigl( \tan t, \tfrac{x}{\cos t} 
\Bigr) e^{-\fr{i |x|^2 \tan t}{2}}.
\end{align*}
The energy-critical counterpart 
to~\eqref{e:intro_m-crit_v} with 
$V = 
\tfr{1}{2}|x|^2$ was recently 
studied by the first author~\cite{me_quadratic_potential}. 

While the Lens transform may be inverted to deduce global wellposedness for the 
mass-critical
harmonic oscillator when $\omega_j \equiv \tfr{1}{2}$, this miraculous
connection with 
equation~\eqref{e:intro_m-crit} evaporates as soon as the $\omega_j$ are not 
all 
equal. Studying the equation in greater generality therefore requires a 
more 
robust line of attack,
such as the concentration-compactness and rigidity paradigm. 
To implement that strategy one needs appropriate inverse $L^2$ Strichartz
estimates.  This is no 
small matter since the Fourier-analytic techniques underpinning the 
proofs of the constant-coefficient theorems---most notably, Fourier 
restriction estimates---are ill-adapted to large 
variable-coefficient perturbations. 

We present a different approach to these inverse estimates in one space 
dimension. By 
eschewing Fourier analysis for physical space arguments, we can 
treat a family of Schr\"{o}dinger operators that includes the free particle and 
the harmonic oscillator. Moreover, our potentials are allowed to depend on time.

	\subsection{The setup}
	Consider a (possibly time-dependent) Schr\"{o}dinger operator on the real 
	line
	\[
	H(t) = -\tfrac{1}{2} \partial^2_x + V(t, x) \quad x \in \mf{R},
	\]
	and assume $V$ is a subquadratic potential.  Specifically, we require that $V$ satisfies the following hypotheses:
	\begin{itemize}
		\item For each $k \ge 2$, there exists there exists $M_k < \infty$  so 
		that
		\begin{equation}
		\label{e:V_h1}
		\|V(t, x)\|_{L^\infty_t L^\infty_x( |x| \le 1)} + \| \partial^k_x 
		V(t, x) \|_{L^\infty_{t,x}} +
		\| \partial^k_x \partial_t V(t, x)\|_{L^\infty_{t,x}} \le M_k.
		\end{equation}
		\item There exists some $\varepsilon > 0$ so that
		\begin{equation}
		\label{e:V_h2}
		| \langle x \rangle^{1+\varepsilon} \partial^3_x V| + | \langle x
		\rangle^{1+\varepsilon}\partial^3_x \partial_t V| \in L^\infty_{t,x}.
		\end{equation}
		By the fundamental theorem of calculus, this implies that the second
		derivative $\partial^2_x V(t, x)$ converges as $x \to \pm
		\infty$.
		Here and in the sequel, we write
		$\langle x \rangle := (1 + |x|^2)^{1/2}$.
	\end{itemize}
	Note that the potentials $V = 0$ and $V = \tfr{1}{2}x^2$ both fall into 	this class. 
	
	The first set of conditions on the space derivatives of $V$ are quite 
	natural 
	in view of classical Fourier integral operator constructions, from which 
	one 
	can deduce dispersive and Strichartz estimates; see 
	Theorem~\ref{t:propagator}. We also need some time regularity of solutions 
	for our spacetime orthogonality arguments.
	However, the decay hypothesis on the third derivative $\partial^3_x V$ is 
	technical; see 
	the 
	discussion surrounding Lemma~\ref{l:technical_lma_2} 
	below. 
	
	The propagator $U(t, s)$ for such Hamiltonians is known
	to obey Strichartz estimates at least locally in time:
	\begin{equation}
	\label{e:loc_str}
	\| U(t, s) f\|_{L^6_{t,x} (I \times \mf{R})} \lesssim_{I} 
	\|f\|_{L^2(\mf{R})}
	\end{equation}
	for any compact interval $I$ and any fixed $s \in \mf{R}$; see 
	Corollary~\ref{c:strichartz}. Note that $U(t, s) = 
	e^{-i(t-s)H}$ is a one-parameter group if one assumes that $V = V(x)$ is 
	time-independent, but our methods do not require this assumption.
	
	Our main result asserts that if 
	the 
	left side is nontrivial relative to the right side, then the evolution of initial data 
	must contain a ``bubble" of concentration. Such
	concentration will be detected by probing the solution with suitably 
	scaled, 
	translated, and
	modulated test functions.
	
	For $\lambda > 0$ and
	$(x_0, \xi_0) \in T^* \mf{R} \cong \mf{R}_x \times \mf{R}_\xi$,
	define the scaling and phase space translation operators
	\[
	S_\lambda f (x) = \lambda^{-1/2} f(\lambda^{-1} x) \quad\text{and}\quad \pi(x_0, \xi_0) f 
	(x) 
	= e^{i(x-x_0)\xi_0} f(x - x_0).
	\]
	Let $\psi$ denote a real even Schwartz function with $\|\psi\|_2 = (2\pi)^{-1/2}$. 
	Its 
	phase space translate
	$\pi(x_0, \xi_0) \psi$ is localized in space near $x_0$ and in
	frequency near $\xi_0$.

	\begin{thm}
		\label{t:inv_str}
		There exists $\beta > 0$ such that if
		$0 < \varepsilon \le \| U(t, 0) f\|_{L^6([-\fr{1}{2}, \fr{1}{2}]
			\times \mf{R})}$ and $\| f\|_{L^2} \le A$, then
		\[
		\sup_{z \in T^* \mf{R}, \ 0 < \lambda \le 1, \ |t| \le 1/2} |
		\langle \pi(z) S_\lambda \psi, U(t, 0) f \rangle_{L^2(\mf{R})}| \ge C
		\varepsilon (\tfr{\varepsilon}{A} )^{\beta}
		\]
		for some constant $C$ depending on the seminorms in \eqref{e:V_h1} 
		and~\eqref{e:V_h2}.
	\end{thm}
	
	By repeatedly applying the following corollary, one can obtain a linear 
	profile decomposition. For simplicity, we state it assuming the
	potential is time-independent (so that $U(t, 0) = e^{-itH}$).
	\begin{cor}
		\label{c:lpd_profile_extraction}
		Let $\{ f_n\} \subset L^2(\mf{R})$ be a sequence such that
		$0 < \varepsilon \le \| e^{-itH} f_n\|_{L^6_{t,x} ([-\fr{1}{2},
			\fr{1}{2}] \times \mf{R})}$ and $\|f\|_{L^2} \le A$
		for some constants $A, \varepsilon > 0$. Then, after passing to a
		subsequence, there exist a sequence of parameters
		\[\{(\lambda_n, t_n, z_n)\}_n \subset (0, 1] \times [-1/2, 1/2]
		\times T^* \mf{R}\] and a function $0 \ne \phi \in L^2$ such that,
		
		\begin{gather}
		S_{\lambda_n}^{-1} \pi(z_n)^{-1}  e^{-it_n H} f_n \rightharpoonup \phi
		\text{ in } L^2 \nonumber \\
		\| \phi\|_{L^2} \gtrsim \varepsilon (\tfr{\varepsilon}{ A} )^{\beta}.
		\label{e:profile_nonzero}
		\end{gather}
		Further,
		\begin{gather}
		\| f_n\|_2^2 - \| f_n - e^{it_n H} \pi(z_n) S_{\lambda_n} \phi
		\|_2^2 - \| e^{it_n H} \pi(z_n) S_{\lambda_n} \phi\|_2^2 \to 0. 
		\label{e:L2_decoupling}
		\end{gather}
	\end{cor}
	
	\begin{proof}
		By Theorem~\ref{t:inv_str}, there exist $(\lambda_n, t_n, z_n)$ such
		that
		$| \langle \pi(z_n) S_{\lambda_n} \psi, e^{-it_n H}f_n \rangle|
		\gtrsim \varepsilon (\tfr{\varepsilon}{ A} )^{\beta}$.
		As the sequence $S_{\lambda_n}^{-1} \pi(z_n)^{-1} e^{-it_n H}f_n$ is
		bounded in $L^2$, it has a weak subsequential limit $\phi \in
		L^2$. Passing to this subsequence, we have
		\[
		\| \phi\|_2 \gtrsim | \langle \psi, \phi \rangle| = \lim_{n \to \infty} |
		\langle \psi, S_{\lambda_n}^{-1} \pi(z_n)^{-1} e^{-it_n H} f_n
		\rangle| \gtrsim \varepsilon (\tfr{\varepsilon}{A})^{\beta}.
		\]
		To obtain~\eqref{e:L2_decoupling}, write the left side as
		\[
		\begin{split}
		2 \opn{Re} \langle f_n - e^{it_n H} \pi(z_n) S_{\lambda_n} \phi,
		e^{it_n H} \pi(z_n) S_{\lambda_n} \phi \rangle = 2 \opn{Re} \langle
		S_{\lambda_n}^{-1} \pi(z_n)^{-1} e^{-it_n H} f_n - \phi, \phi
		\rangle \to 0,
		\end{split}
		\]
		by the definition of $\phi$.
	\end{proof}
	
	The restriction to a compact time interval in the above statements is
	dictated by the generality of our hypotheses. For a generic
	subquadratic potential, the $L^6_{t,x}$ norm of a solution need not be
	finite on $\mf{R}_t \times \mf{R}_x$. For example, solutions to the 
	harmonic oscillator (for which $V(x) = \frac12 x^2$) are periodic in time. However, the conclusions may be
	strengthened in some cases. In particular, our methods
	specialize to the case $V = 0$ to yield
	\begin{thm}
		\label{t:inv_str_free_particle}
		If $0 < \varepsilon \le \|e^{\fr{it\Delta}{2}} f\|_{L^6_{t,x}(\mf{R}
			\times \mf{R})} \lesssim \|f\|_{L^2} = A$, then
		\[
		\sup_{z \in T^* \mf{R}, \ \lambda > 0, \ t \in \mf{R}} | \langle \pi(z) 
		S_\lambda \psi, 
		e^{\fr{it \Delta}{2}} f \rangle| \gtrsim \varepsilon 
		(\tfr{\varepsilon}{A})^{\beta}.
		\]
	\end{thm}
	This yields an analogue to Corollary~\ref{c:lpd_profile_extraction},
	which can be used to derive a linear profile decomposition for the one dimensional
	free particle. Such a
	profile decomposition was obtained originally by
	Carles and Keraani~\cite{carles-keraani} using different methods.

	\subsection{Ideas of proof}
	We shall assume in the sequel
	that the initial data $f$ is Schwartz. This assumption will justify
	certain applications of Fubini's theorem and may be removed a
	posteriori by an approximation argument. Further, we prove the theorem with 
	the 
	time interval
	$[-\tfr{1}{2}, \tfr{1}{2}]$ replaced by $[-\delta_0, \delta_0]$, where
	$\delta_0$ is furnished by Theorem~\ref{t:propagator} 
	according to the seminorms $M_k$ of the potential. Indeed, the
	interval $[-\tfr{1}{2}, \tfr{1}{2}]$ can then be tiled by subintervals
	of length $\delta_0$.
		
	Given these preliminary reductions, we describe the
	main ideas of the proof of Theorem~\ref{t:inv_str}. Our goal is to locate
	the parameters describing a concentration bubble in the evolution of the initial
	data. The relevant parameters are a length scale
	$\lambda_0$, spatial center $x_0$, frequency center $\xi_0$, and a time 
	$t_0$ describing when the concentration occurs. Each parameter is 
	associated with a noncompact symmetry or
	approximate symmetry of the Strichartz inequality. For instance, when
	$V = 0$ or $V = \tfr{1}{2}x^2$, both sides of~\eqref{e:loc_str} are
	preserved by translations $f \mapsto f(\cdot - x_0)$ and modulations
	$f \mapsto e^{i(\cdot) \xi_0} f$ of the initial data, while more general 
	$V$ admit an approximate Galilean invariance; see
	Lemma~\ref{l:galilei} below.

	The existing approaches to inverse Strichartz inequalities for the free
	particle can be roughly summarized as follows. First, one uses
	Fourier analysis to isolate a scale $\lambda_0$ and frequency
	center $\xi_0$.  For example, Carles-Keraani prove in their
	Proposition 2.1 that for some $1 < p < 2$,
	\[
	\| e^{it \partial_x^2 } f\|_{L^6_{t,x}(\mf{R} \times \mf{R})}
	\lesssim_p \Bigl( \sup_{J} |J|^{\fr{1}{2} - \fr{1}{p}} \|
	\hat{f}\|_{L^p (J)} \Bigr)^{1/3} \|f\|_{L^2(\mf{R})},
	\]
	where $J$ ranges over all intervals and $\hat{f}$ is the Fourier
	transform of $f$. Then one uses a separate argument to determine $x_0$
	and $t_0$. This strategy ultimately relies on the fact that the
	propagator for the free particle is diagonalized by the Fourier
	transform. 
	
	One does not enjoy that luxury with general Schr\"{o}dinger operators as the
	momenta of particles may vary with time and in a position-dependent
	manner. Thus it is natural to consider the position and frequency
	parameters together in phase space. To this end, we use a 
	wavepacket 
	decomposition as
	a partial substitute for the Fourier transform.	
	Unlike the Fourier transform, however, the wavepacket transform
	requires that one first chooses a length scale. This is not entirely trivial
	because the Strichartz inequality~\eqref{e:loc_str} which
	we are trying to invert has no intrinsic length scale; the
	rescaling
	\[f \mapsto \lambda^{-d/2} f(\lambda^{-1} \cdot), \ 0 < \lambda \ll 1
	\]
	preserves both sides of the inequality exactly when $V = 0$ and
	at least approximately for subquadratic potentials $V$.
		
	We obtain the parameters in a different order. Using
	a direct physical space argument, we show that if $u(t,x)$ is a
	solution with nontrivial $L^6_{t,x}$ norm, then there exists a time
	interval $J$ such that $u$ is large in $L^q_{t,x}(J \times \mf{R})$
	for some $q < 6$. Unlike the $L^6_{t,x}$ norm, the $L^q_{t,x}$ is not 
	scale-invariant, hence the interval $J$ identifies a significant time $t_0$ 
	and physical scale $\lambda_0 = \sqrt{|J|}$. By an interpolation and 
	rescaling argument, we then
	reduce matters to a refined $L^2_x \to L^4_{t, x}$ estimate. This is
	then proved using a wavepacket decomposition, integration by parts,
	and analysis of bicharacteristics, revealing the parameters
	$x_0$ and~$\xi_0$ simultaneously.

	This paper is structured as follows. Section~\ref{s:prelim} collects
	some preliminary definitions and lemmas. The
	heart of the argument is presented in Sections~\ref{s:inv_hls}
	and~\ref{s:L4}.  
	As the identification of a time interval works in any
	number of space dimensions, Section~\ref{s:inv_hls} is written for a 
	general subquadratic
	Schr\"{o}dinger operator on $\mf{R}^d$. In fact the argument there applies 
	to any linear propagator that satisfies the dispersive estimate. In the 
	later sections we specialize to $d = 1$. 
	
	Further insights seem to be needed in two or more space dimensions. A 
	naive attempt to extend our methods to higher dimensions would require us to prove a refined $L^p$
	estimate for some $2 < p < 4$; our arguments in this paper exploit
	the fact that $4$ is an even integer. There is also a more conceptual 
	barrier: while a timescale should serve as a proxy for \emph{one} spatial 
	scale, there may \emph{a priori} exist more than one interesting physical 
	scale in higher 
	dimensions. For instance, the nonelliptic Schr\"{o}dinger equation 
	\begin{align*}
	i\partial_t u = -\partial_x \partial_y u
	\end{align*}
	in two dimensions satisfies the Strichartz estimate~\eqref{e:intro_str} and 
	admits the scaling symmetry $u \mapsto u(t, \lambda x, 
	\lambda^{-1} y)$ in addition to the usual one. A refinement of the 
	Strichartz inequality for 
	this particular example was obtained using Fourier-analytic methods by 
	Rogers and Vargas~\cite{Rogers2006}.  Any higher-dimensional 
	generalization of our methods must somehow distinguish the elliptic and 
	nonelliptic cases.
	
	\subsection*{Acknowledgements} 
	
	This work
	was partially supported by NSF grants DMS-1265868 (PI R. Killip),
	DMS-1161396, and DMS-1500707 (both PI M. Visan).

	\section{Preliminaries}
	\label{s:prelim}

\subsection{Phase space transforms}
	We briefly recall the (continuous) wavepacket decomposition; see for
	instance~\cite{folland}. Fix a real, even Schwartz function
	$\psi \in \mcal{S}(\mf{R}^d)$ with $\| \psi\|_{L^2} = (2\pi)^{-d/2}$.  For
	a function $f \in L^2(\mf{R}^d)$ and a point
	$z = (x, \xi) \in T^* \mf{R}^d = \mf{R}^d_x \times \mf{R}^d_{\xi}$ in phase 
	space,
	define
	\[
	T f(z) = \int_{\mf{R}^d} e^{i(x-y)\xi} \psi(x-y) f(y) \, dy = \langle f, 
	\psi_z \rangle_{L^2(\mf{R}^d)}.
	\]
	By taking the Fourier transform in the $x$ variable, we get
	\[
	\mcal{F}_x Tf (\eta, \xi) = \int_{\mf{R}^d} e^{-iy\eta}
	\hat{\psi}(\eta - \xi) f(y) \, dy = \hat{\psi}(\eta - \xi) \hat{f}(\eta).
	\]
	Thus $T$ maps
	$\mcal{S}(\mf{R}^d) \to \mcal{S}(\mf{R}^d \times \mf{R}^d)$ and is an
	isometry $L^2 (\mf{R}^d) \to L^2( T^* \mf{R}^d)$. The hypothesis that
	$\psi$ is even implies the adjoint formula
	\[
	T^* F(y) = \int_{T^* \mf{R}^d} F(z) \psi_z(y) \, dz
	\]
	and the inversion formula
	\[
	f = T^*T f = \int_{T^* \mf{R}^d} \langle f, \psi_z
	\rangle_{L^2(\mf{R}^d)} \psi_z \, dz.
	\]

\subsection{Estimates for bicharacteristics}
Let $V(t,x)$ satisfy $\partial_x^k V(t, \cdot) \in L^\infty(\mf{R}^d)$
for all $k \ge 2$, uniformly in $t$. The
time-dependent symbol $h(t, x, \xi) = \tfr{1}{2}|\xi|^2 + V(t,x)$ defines a 
globally Lipschitz
Hamiltonian vector field
$\xi \partial_x - (\partial_x V) \partial_\xi$ on $T^* \mf{R}^d$, hence the 
flow map $\Phi(t, s) : T^* \mf{R}^d \to T^* \mf{R}^d$ is well-defined for all 
$s$ and $t$.
For $z = (x, \xi)$, write
$z^t = ( x^t(z), \xi^t(z) ) = \Phi(t, 0)(z)$ denote the bicharacteristic 
starting from $z$ at time $0$.

Fix
$z_0, z_1 \in T^* \mf{R}^d$. We obtain by integration
\[
\begin{split}
x_0^t - x_1^t &= x_0^s - x_1^s + (t-s) (\xi_0^s - \xi_1^s) - \int_s^t
(t - \tau) ( \partial_x V(\tau, x_0^\tau) - \partial_x V(\tau,
x_1^\tau)) \, d\tau\\
\xi_0^t - \xi_1^t &= \xi_0^s - \xi_1^s - \int_s^t (\partial_x V(\tau,
x_0^\tau) - \partial_x V (\tau, x_1^\tau) ) \, d \tau.
\end{split}
\]
As
$|\partial_x V (\tau, x_0^\tau) - \partial_x V (\tau, x_1^\tau)| \le
\| \partial^2_x V \|_{L^\infty} |x_0^\tau - x_1^\tau|$,
we have for $|t-s| \le 1$
\begin{equation}
\label{e:difference_of_trajectories}
	\begin{split}
&|x_0^t - x_1^t| \le ( |x_0^s - x_1^s| + |t-s| |\xi_0^s -
\xi_1^s|) e^{ \| \partial_x^2 V\|_{L^\infty} },\\
&| \xi_0^t - \xi_1^t - (\xi_0^s - \xi_1^s)| \le ( |t-s| |x_0^s -
x_1^s| + |t-s|^2 |\xi_0^s - \xi_1^s|) \|\partial^2_x V\|_{L^\infty} 
e^{ \| \partial_x^2 V\|_{L^\infty} } ,\\
&|x_0^t - x_1^t - (x_0^s - x_1^s) - (t-s) (\xi_0^s - \xi_1^s)|
\le (|t-s|^2 |x_0^s - x_1^s| + |t-s|^3 |\xi_0^s - \xi_1^s|) e^{\| 
	\partial^2_x V \|_{L^\infty}}.
	\end{split}
\end{equation}
In the sequel, we shall always assume that $|t-s| \le 1$, and all implicit 
constants shall depend on $\partial^2_x V$ or finitely many 
higher derivatives. We 
also
remark that this time restriction may be dropped if $\partial^2_x V \equiv 
0$ (as in Theorem~\ref{t:inv_str_free_particle}). 
The preceding
computations immediately yield the following dynamical
consequences:
\begin{lma}
	\label{l:collisions}
	Assume the preceding setup.
	\begin{itemize}
		\item There exists $\delta>0$, depending on $\|\partial_x^2 V 
		\|_{L^\infty}$, 
		such that 
		$|t-s| \le \delta$ implies
		\[
		|x^t_0 - x_1^t - (x_0^s - x_1^s) - (t-s) (\xi_0^s - \xi_1^s)| \le
		\fr{1}{100} ( |x_0^s - x_1^s| + |t-s| |\xi_0^s - \xi_1^s| ).
		\]
		Hence if $|x_0^s - x_1^s | \le r$ and $C
		\ge 2$, then
		$|x_0^t - x_1^t| \ge Cr$ for
		$\tfr{2Cr}{|\xi_0^s - \xi_1^s|} \le |t-s| \le \delta$. Informally,
		two particles colliding with sufficiently large relative velocity 
		will
		interact only once during a length $\delta$ time interval.
		
		\item With $\delta$ and $C$ as above, if $|x_0^s - x_1^s| \le r$, then
		\[
		| \xi_0^t - \xi_1^t - (\xi_0^s - \xi_1^s)| \le \min \Bigl( \delta,
		\fr{2Cr}{|\xi_0^s - \xi_1^s|} \Bigr) Cr \|\partial^2_x
		V\|_{L^\infty} e^{\|\partial^2_x V\|_{L^\infty}}
		\]
		for all $t$ such that $|t-s|\leq \delta$ and $|x_0^{\tau} - x_1^{\tau} | \le Cr$ for all $s\leq \tau\leq t$. That is, the
		relative velocity of two particles remains essentially constant
		during an interaction.
	\end{itemize}
\end{lma}

The following technical lemma will be used in 
Section~\ref{s:L4_estimate}.
\begin{lma}
	\label{l:technical_lma}
	There exists a constant $C > 0$ so that if
	$Q_\eta = (0, \eta) +[-1, 1]^{2d}$ and $r \ge 1$, then
	\[
	\bigcup_{|t - t_0| \le \min(|\eta|^{-1}, 1) } \Phi(t, 0)^{-1} (z_0^t + 
	rQ_\eta) \subset
	\Phi(t_0, 0)^{-1}  (z_0^{t_0}+ Cr Q_\eta).
	\]
\end{lma}
In other words, if the bicharacteristic $z^t$ starting at $z \in 
T^*\mf{R}^d$ passes
through the cube $z_0^t + rQ_{\eta}$ in phase space during some time
window $|t - t_0| \le \min(|\eta|^{-1},1)$, then it must lie in the
dilate $z_0^{t_0} + CrQ_{\eta}$ at time $t_0$.
\begin{proof}
	If $z^{s} \in z_0^{s} + rQ_\eta$,
	then~\eqref{e:difference_of_trajectories} and
	$|t-s| \le \min(|\eta|^{-1}, 1)$ imply that
	
	\begin{gather*}
	|x^t - x_0^t| \lesssim |x^s - x_0^s| + \min(|\eta|^{-1}, 1) (|\eta| + 
	r) \lesssim r,\\
	|\xi^t - \xi_0^t - (\xi^s - \xi_0^s)| \lesssim r\min(|\eta|^{-1}, 1).
	\end{gather*}	
\end{proof}

\subsection{The Schr\"{o}dinger propagator}
In this section we recall some facts regarding the quantum
propagator for subquadratic potentials. First, we have the following 
oscillatory integral representation:
\begin{thm}[{Fujiwara~\cite{fujiwara_fundamental_solution,fujiwara_path_integrals}}]
	\label{t:propagator}
	Let $V(t, x)$ satisfy
	\[
	M_k := \| \partial^k_x V(t, x) \|_{L^\infty_{t,x}} +\| V(t, 
	x)\|_{L^\infty_t 
		L^\infty_x ( |x| \le 1)}  < \infty
	\]
	for all $k \ge 2$. There exists a constant $\delta_0 > 0$ such that for 
	all
	$0 < |t-s| \le \delta_0$ the propagator $U(t, s)$ for
	$H = -\tfr{1}{2} \Delta + V(t, x)$ has Schwartz kernel
	\[
	U(t, s) (x, y) = \Bigl (\fr{1}{2\pi i (t-s)} \Bigr )^{d/2} a(t, s, x, 
	y) e^{iS(t,
		s, x, y)},
	\]
	where for each $m > 0$ there is a constant $\gamma_m > 0$ such that
	\[
	\| a(t, s, x, y) - 1 \|_{C^m (\mf{R}^d_x \times \mf{R}^d_y)} \le 
	\gamma_m |t-s|^2.
	\]
	Moreover
	\[
	S(t, s, x, y) =  \fr{|x-y|^2}{2(t-s)} + (t-s) r(t, s, x, y),
	\]
	with
	\[
	|\partial_xr| + |\partial_yr| \le C(M_2) (1 + |x| + |y|),
	\]
	and for each multindex $\alpha$ with $|\alpha| \ge 2$, the quantity
	\begin{equation}
	\nonumber
	C_\alpha = \| \partial^\alpha_{x, y} r(t, s, \cdot, \cdot ) 
	\|_{L^\infty}
	\end{equation}
	is finite. The map $U(t, s) : \mcal{S}(\mf{R}^d) \to
	\mcal{S}(\mf{R}^d)$ is a topological isomorphism, and all implicit
	constants depend on finitely many $M_k$. 
	
\end{thm}

\begin{cor}[Dispersive and Strichartz estimates]
	\label{c:strichartz}
	If $V$ satisfies the hypotheses of the previous theorem, then $U(t,
	s)$ admits the fixed-time bounds
	\[
	\|U(t,s)\|_{L^1_x(\mf{R}^d) \to L^\infty_x(\mf{R}^d)} \lesssim 
	|t-s|^{-d/2}
	\]
	whenever $|t-s| \le \delta_0$. For any compact time interval $I$ and
	any exponents $(q,r)$ satisfying $2 \le q, r \le \infty$, $\tfr{2}{q} + 
	\tfr{d}{r} = \tfr{d}{2}$, and $(q, r, d) \ne (2, \infty, 2)$, we have
	\[
	\| U(t, s) f\|_{L^{q}_t L^r_x ( I \times \mf{R}^d )}
	\lesssim_{I} \|f\|_{L^2(\mf{R}^d)}.
	\]
\end{cor}
\begin{proof}
	Combining Theorem~\ref{t:propagator} with the general machinery of Keel-Tao~\cite{keel-tao}, we obtain
	\[
	\| U(t, s) f\|_{ L^q_t L^r_x ( \{ |t-s| \le \delta_0 \} \times
		\mf{R}^d)} \lesssim \|f\|_{L^2}.
	\]
	If $I = [T_0, T_1]$ is a general time interval, partition it into
	subintervals $[t_{j-1}, t_{j}]$ of length at most $\delta_0$. For each
	such subinterval we can write $U(t, s) = U(t, t_{j-1}) U(t_{j-1}, s)$, 
	thus
	\[
	\| U(t, s) f\|_{L^q_t L^r_x ( [t_{j-1}, t_{j}] \times \mf{R}^d)}
	\lesssim \|U (t_{j-1}, s) f\|_{L^2} = \|f\|_{L^2}.
	\]
	The corollary follows from summing over the subintervals.
\end{proof}

Recall that solutions to the free particle equation
$i\partial_t u = -\tfr{1}{2}\Delta u$ with $ u(0) = \phi$ transform
as follows under phase space
translations of the initial data:
\begin{equation}
\label{e:galilei_free_particle}
e^{\fr{it\Delta}{2}} \pi(x_0, \xi_0) \phi(x) = e^{i[(x-x_0)\xi_0 - 
	\fr{1}{2}t
	|\xi_0|^2]} (e^{\fr{it\Delta}{2}} \phi)(x - x_0 - t\xi_0).
\end{equation}
Physically, $\pi(x_0, \xi_0)\phi$ represents the state of a quantum
particle with position $x_0$ and momentum $\xi_0$. The above
relation states that the time evolution of $\pi(x_0, \xi_0)\phi$ in
the absence of a potential oscillates in space and time at frequency
$\xi_0$ and $-\tfr{1}{2}|\xi_0|^2$, respectively, and tracks the
classical trajectory $t \mapsto x_0 + t \xi_0$.

In the presence of a potential, the time evolution of such modified
initial data admits an analogous description:
\begin{lma}
	\label{l:galilei}
	If $U(t, s)$ is the propagator for
	$H = -\tfr{1}{2}\Delta + V(t, x)$, then
	\[
	\begin{split}
	U(t,s) \pi(z_0^s) \phi(x) &=  e^{i [ (x-x_0^t) \xi_0^t + \int_s^t
		\fr{1}{2} |\xi_0^\tau|^2 - V(\tau, x_0^\tau) \, d \tau]} U^{z_0}(t,
	s) \phi (x - x_0^t)
	\\ &=e^{i \alpha(t, s, z_0)} \pi(
	z_0^t) [U^{z_0} (t, s) \phi](x),
	\end{split}
	\]
	where 
	\[
	\alpha(t, s, z) = \int_s^t \fr{1}{2} |\xi_0^\tau|^2 - V(\tau, x_0^\tau) 
	\, 
	d\tau
	\]
	is the classical action, $U^{z_0}(t, s)$ is the propagator for $H^{z_0} 
	=
	-\tfr{1}{2} \Delta + V^{z_0} (t, x)$ with
	\[
	V^{z_0}(t, x) = V(t, x_0^t + x) - V(t,  x_0^t) -  x \partial_x V (t,
	x_0^t) = \langle x, Qx \rangle
	\]
	where  
	\[
	Q(t,x) = \int_0^1 (1-\theta) \partial^2_x V(t, x_0^t + \theta x) \,
	d\theta,
	\]
	and $z_0^t = (x_0^t, \xi_0^t)$ is the trajectory of $z_0$ under the
	Hamiltonian flow of the symbol $h = \tfr{1}{2} |\xi|^2 + V(t, x)$.
	The propagator $U^{z_0}(t, s)$ is continuous on $\mcal{S}(\mf{R}^d)$
	uniformly in $z_0$ and $|t-s| \le \delta_0$.
\end{lma}

\begin{proof}
	The formula for $U(t, s) \pi(z_0^s) \phi$ is verified by direct
	computation. To obtain the last statement, we notice that
	$\|\partial^k_x V^{z_0}\|_{L^\infty} = \| \partial^k_x V
	\|_{L^\infty}$
	for $k \ge 2$, and appeal to the last part of
	Theorem~\ref{t:propagator}.
\end{proof}

\begin{rmks}\leavevmode\\[-5mm]
	\begin{itemize}
		\item Lemma~\ref{l:galilei} reduces to~\eqref{e:galilei_free_particle} when $V = 
		0$ and gives analogous formulas when $V$ is a polynomial of degree at 
		most
		$2$. When $V = Ex$ is the potential for a constant electric field, 
		we recover the Avron-Herbst
		formula by setting $z_0 = 0$ (hence $V^{z_0} = 0$). For $V = \sum_j 
		\omega_j 
		x_j^2$, we get the generalized Galilean symmetry 
		mentioned in the introduction.
		
		\item Direct computation shows that the above identity extends to 
		semilinear
		equations of the form
		\[
		i\partial_t u = (-\tfrac{1}{2}\Delta + V) u \pm |u|^p u.
		\]
		That is, if $u$ is the solution with $u(0) = \pi(z_0) \psi$, then
		\[
		u(t) = e^{i\int_0^t \fr{1}{2} |\xi_0^\tau|^2 - V(\tau, x_0^\tau) \, 
		d\tau} \pi(z_0^t) 
		u_{z_0}(t)
		\]
		where $u_{z_0}$ solves
		\[
		i\partial_t u_{z_0} = \bigl(-\tfrac{1}{2} \Delta + V^{z_0}  \bigr)
		u_{z_0} \pm |u_{z_0}|^p u_{z_0} \quad\text{with}\quad u_{z_0}(0) = \psi,
		\]
		where the potential $V^{z_0}$ is as defined in Lemma~\ref{l:galilei}. 
		
		\item One can combine this lemma with a wavepacket decomposition to
		represent a solution $U(t, 0) f$ as a sum of wavepackets
		\[
		U(t, 0)f = \int_{z_0 \in T^* \mf{R}^d} \langle f , \psi_{z_0} 
		\rangle U(t, 0)(\psi_{z_0})  \, dz_0,
		\]
		where the oscillation of each wavepacket $U(t, 0)(\psi_{z_0})$
		is largely captured in the phase
		\[
		(x-x_0^t)\xi_0^t + \int_0^t \fr{1}{2}|\xi_0^\tau|^2 - V(\tau, 
		x_0^\tau) \, d\tau.
		\]
		Our arguments will make essential use of this information. Analogous 
		wavepacket representations have been
		constructed by Koch and Tataru 
		for a broad
		class of pseudodifferential operators; see~\cite[Theorem 
		4.3]{KochTataru2005} and its proof.
		
	\end{itemize}
\end{rmks}

\section{Locating a length scale}
\label{s:inv_hls}

	The first step in the proof of Theorem~\ref{t:inv_str} is to identify both 
	a characteristic time 
	scale and
	temporal center for our sought-after bubble of concentration. Recall
	that the usual $TT^*$ proof of the non-endpoint Strichartz inequality
	combines the dispersive estimate with the Hardy-Littlewood-Sobolev
	inequality in time. By using a refinement of the latter, one
	can locate a time interval on which the solution is large in a
	non-scale invariant spacetime norm.
	
	
	\begin{prop}
		\label{p:inv_hls}
		Consider a pair $(q, r)$ in Corollary~\ref{c:strichartz} with $2 < q < 
		\infty$, and suppose
		$u = U(t, 0)f$ solves
		\[
		i\partial_t u = \bigl ( -\tfrac{1}{2}\Delta + V \bigr )u \quad\text{with}\quad u(0) = f \in 
		L^2(\mf{R}^d)
		\]
		with $\|f\|_{L^2 (\mf{R}^d)} = 1$ and
		$\|u\|_{L^q_t L^r_x ([-\delta_0, \delta_0] \times \mf{R}^d)} \ge
		\varepsilon$,
		where $\delta_0$ is the constant from Theorem~\ref{t:propagator}. 
		Then
		there is a time interval $J\subset [-\delta_0, \delta_0]$ such that
		\[
		\|u\|_{L^{q-1}_t L^r_x (J \times \mf{R}^d)} \gtrsim 
		|J|^{\fr{1}{q(q-1)}} \varepsilon^{\fr q{q-2}}.
		\]
	\end{prop}
	
	\begin{rmk}
		That this estimate singles out a special length scale is easiest to see
		when $V = 0$. For ease of notation, suppose $J=[0,1]$ in Proposition~\ref{p:inv_hls}.  	As $\|u\|_{L^q_t L^r_x(\mf{R} \times \mf{R}^d)} \lesssim \|f\|_{L^2} <
		\infty$, for each $\eta > 0$ there exists $T > 0$ so that (suppressing the
		region of integration in $x$) $\|u\|_{L^q_t L^r_x (\{ |t| \geq T\})} < \eta$.  With $u_{\lambda}(t, x) = \lambda^{-d/2} u(\lambda^{-2}t, \lambda^{-1} x)
		= e^{\fr{it\Delta}{2}} (f_\lambda)$ where $f_\lambda = \lambda^{-d/2} f(\lambda^{-1})$, we have
		\[
		\begin{split}
		\|u_{\lambda} \|_{L^{q-1}_t L^r_x ([0, 1])} &\le \|
		u_{\lambda}\|_{L^{q-1}_t L^r_x ([0, \lambda^2 T]) }+ \| u_{\lambda}
		\|_{L^{q-1}_t L^r_x ([\lambda^2 T, 1])} \\
		&\le (\lambda^2
		T)^{\fr{1}{q(q-1)}} \| u_\lambda \|_{L^q_t L^r_x ( [0, \lambda^2
			T])} + \|u_\lambda \|_{L^q_t L^r_x ([ \lambda^2 T, 1])}\\
		&\le (\lambda^2 T)^{\fr{1}{q(q-1)}} \|u\|_{L^q_t L^r_x} + \eta,
		\end{split}
		\]
		which shows that
		\[
		\| u_{\lambda} \|_{L^{q-1}_t L^r_x ([0, 1] \times \mf{R}^d)} \to 0
		\text{ as } \lambda \to 0.
		\]
		Thus, Proposition~\ref{p:inv_hls} shows that concentration of the solution cannot occur at arbitrarily small scales. Similar
		considerations preclude $\lambda \to \infty$.
	\end{rmk}
	
	We shall need the following inverse
	Hardy-Littlewood-Sobolev estimate. For $0 < s < d$, denote by $I_s f(x)
	= (|D|^{-s} f) (x) = c_{s,d} \int_{\mf{R}^d} \fr{f(x-y)}{|y|^{d-s}} \, dy$
	the fractional integration operator.

	\begin{lma}[Inverse HLS] Fix $d\geq 1$, $0<\gamma<d$, and $1<p<q<\infty$ 
	obeying
		$\tfrac dp = \tfrac dq+d - \gamma$.
		If $f\in L^p(\R^d)$ is such that
		$$
		\| f \|_{L^p(\R^d)} \leq 1 \qtq{and} \| |x|^{-\gamma} * f \|_{L^q} \geq 
		\eps,
		$$
		then there exists $r>0$ and $x_0\in \R^d$ so that
		\begin{equation}\label{E:inv con}
		\int_{r<|x-x_0|<2r} |f(x)|\,dx \gtrsim \eps^{\frac{q}{q-p}} 
		r^{\frac{d}{p'}}.
		\end{equation}
	\end{lma}
	
	\begin{proof}
		Our argument is based off the proof of the usual 
		Hardy--Littlewood--Sobolev inequality due to Hedberg 
		\cite{Hedberg1972}; see also \cite[\S VIII.4.2]{stein}.
		
		Suppose, in contradiction to \eqref{E:inv con}, that
		\begin{equation}\label{E:inv con'}
		\sup_{x_0,r} \  r^{-\frac{d}{p'}} \!\int_{r<|x-x_0|<2r} |f(x)|\,dx \leq 
		\eta \eps^{\frac{q}{q-p}}
		\end{equation}
		for some small $\eta=\eta(d,p,\gamma)>0$ to be chosen later.

		As in the Hedberg argument, a layer-cake decomposition of 
		$|y|^{-\gamma}$ yields the following bound in terms of the maximal 
		function:
		$$
		\int_{|y|\leq r} |f(x-y)| |y|^{-\gamma} \,dy \lesssim r^{d-\gamma} 
		[Mf](x) .
		$$
		On the other hand, summing \eqref{E:inv con'} over dyadic shells yields
		$$
		\int_{|y|\geq r} |f(x-y)| |y|^{-\gamma} \,dy \lesssim  
		\eps^{\frac{q}{q-p}} \,\eta\, r^{\frac d{p'}-\gamma} .
		$$
		Combining these two estimates and optimizing in $r$ then yields
		$$
		\biggl| \int_{\R^d} |f(x-y)| |y|^{-\gamma} \,dy \biggr| \lesssim  \eps 
		\eta^{1-\frac pq} | [Mf](x) |^{\frac pq}
		$$
		and hence
		$$
		\| |x|^{-\gamma} * f \|_{L^q} \lesssim \eps \eta^{1-\frac pq} \| Mf 
		\|_{L^p(\R^d)}^{\frac pq} \lesssim \eps \eta^{1-\frac pq} .
		$$
		Choosing $\eta>0$ sufficiently small then yields the sought-after 
		contradiction.
	\end{proof}

	\begin{proof}[Proof of Proposition~\ref{p:inv_hls}]
		Define the map $T: L^2_x \to L^q_t L^r_x$ by $Tf(t) = U(t, 0)
		f$, which by Corollary~\ref{c:strichartz} is continuous. By 
		duality, 
		$\varepsilon \le \| u \|_{L^q_t L^r_x}$ implies
		$\varepsilon \le \|T^* \phi \|_{L^2_x}$, where
		\[
		\phi = \fr{ |u|^{r-2} u}{ \|u(t)\|_{L^r_x}^{r-1}} \fr{ \|u(t)
			\|_{L^r_x}^{q-1}}{ \|u \|_{L^q_t L^r_x}^{q-1}}
		\]
		satisfies $\| \phi\|_{L^{q'}_t L^{r'}_x} = 1$ and
		\[
		T^* \phi = \int U(0, s) \phi(s) \, ds.
		\]
		By the dispersive estimate of Corollary~\ref{c:strichartz},
		\begin{align*}
		\varepsilon^2&\le \langle T^* \phi, T^* \phi \rangle_{L^2_x} = \langle \phi, T	T^* \phi \rangle_{L^2_{t,x}} = \iiint \overline{\phi(t)} U(t,s) \phi(s) \,dx  ds dt 
		\lesssim \iint \fr{ G(t) G(s)} { |t-s|^{2/q}} \, dsdt\\
		& \lesssim \|G\|_{L_t^{q'}} \||t|^{-\frac2q}*G\|_{L_t^q},
		\end{align*}
		where $G(t) = \| \phi(t)\|_{L^{r'}_x}$.  Appealing to the previous lemma with $p = q'$, we derive
		\begin{align*}
		\sup_J |J|^{-\frac1q}\|G\|_{L^1_t(J)}\gtrsim \eps^{\frac{2(q-1)}{q-2}},
		\end{align*}
		which, upon rearranging, yields the claim.
	\end{proof}

\section{A refined $L^4$ estimate}
\label{s:L4}

Now we specialize to the one-dimensional setting $d = 1$. We are particularly 
interested in the Strichartz exponents
\[
(q_0, r_0) = \Bigl( \fr{ 7 + \sqrt{33} }{2}, \fr{5 + \sqrt{33}}{2} \Bigr)
\]
determined by the conditions
$\tfr{2}{q_0} + \tfr{1}{r_0} = \tfr{1}{2}$ and $q_0-1 = r_0$. Note that $5 < r_0 < 
6$.
	
	Suppose $\|f\|_{L^2} = A
	$ and that $u = U(t, 0) f$ satisfies
	$\| u \|_{L^6_{t,x} ( [-\delta_0, \delta_0] 
		\times \mf{R})} = \varepsilon$.
	Using the inequality $\|u\|_{L^6_{t,x}} \le \| u\|_{L^5_t 
	L^{10}_x}^{1-\theta} \| 
	u\|_{L_t^{q_0} L_x^{r_0}}^\theta$ for some $0 < \theta < 1$, estimating the 
	first factor by Strichartz, and applying 
	Proposition~\ref{p:inv_hls}, we find a time interval $J = [t_0 - \lambda^2, 
	t_0 + \lambda^2]$ such that
\[
\| u \|_{L^{q_0-1}_t L^{r_0}_x ( J \times \mf{R})} \gtrsim A
|J|^{\fr{1}{q_0(q_0-1)}} \bigl ( \tfrac{\varepsilon}{A} \bigr )^{
	\fr{q_0}{\theta(q_0-2)}}.
\]
Setting
\[
u(t, x) = \lambda^{-1/2} u_\lambda (\lambda^{-2} (t-t_0),
\lambda^{-1} x),
\]
we get 
\[
i\partial_t u_\lambda = (-\tfr{1}{2} \partial^2_x + V_\lambda
)u_\lambda = 0 \quad\text{with}\quad u_\lambda(0, x) = \lambda^{1/2} u(t_0, \lambda x)
\]
and $V_\lambda(t, x) = \lambda^2 V(t_0 + \lambda^2 t, \lambda x)$ satisfies the hypotheses~\eqref{e:V_h1} and~\eqref{e:V_h2} for 
all
$0 < \lambda \le 1$. By the corollary and a change of variables,
\[
\| u_\lambda \|_{L^{q_0-1}_{t,x} ([-1, 1] \times \mf{R})}
\gtrsim A (\tfr{\varepsilon}{A})^{\fr{q_0}{\theta(q_0-2)}}.
\]

As $4 < q_0-1 < 6$, Theorem~\ref{t:inv_str} will follow by interpolating
between the $L^2_x \to L^6_{t,x}$ Strichartz estimate and the following
$L^2_x \to L^4_{t,x}$ estimate. Recall that $\psi$ is the test 
function 
fixed in the introduction.
\begin{prop}\label{P:4.1}
	Let $V$ be a potential satisfying the hypotheses~\eqref{e:V_h1} 
	and~\eqref{e:V_h2}, and denote by $U_V(t, s)$ the linear propagator. There 
	exists $\delta_0 > 0$ so that if $\eta \in C^\infty_0( (-\delta_0, 
	\delta_0))$, 
	\label{p:refined_L4}
	\[
	\|U_V(t, 0) f\|_{L^4_{t,x}(\eta(t) dx dt)} \lesssim \|f\|_{2}^{1-\beta}
	\sup_z | \langle \psi_z, f \rangle |^{\beta}
	\]
	for some absolute constant $0 < \beta < 1$. 
\end{prop}
Note that this estimate is trivial if the right side is replaced by 
$\|f\|_2$ since $L^4_{t,x}$ is controlled by $L^2_{t,x}$ 
and $L^{6}_{t,x}$, which on a compact time interval are bounded above by $\|f\|_2$ by unitarity and Strichartz, respectively.

	\subsection{Proof of Proposition~\ref{p:refined_L4}}
	\label{s:L4_estimate}

	We fix the potential $V$ and drop the subscript $V$ from the
	propagator. It suffices to prove the proposition for $f\in \mathcal S(\R)$.  Decomposing $f$
	into wavepackets $f = \int_{T^* \mf{R}} \langle f, \psi_z \rangle \psi_z \, dz$ and
	expanding the $L^4_{t,x}$ norm, we get
	\[
	\begin{split}
	\| U(t,0) f\|^4_{L^4_{t, x}(\eta(t) dx dt)}  \le \int_{ (T^* \mf{R})^4} K(z_1, z_2, z_3, 
	z_4)
	\prod_{j=1}^4 |\langle f, \psi_{z_j} \rangle| \, dz_1 dz_2 dz_3 dz_4,
	\end{split}
	\]
	where
	\begin{equation}
	\label{e:kernel_def}
	K (z_1,z_2,z_3,z_4):= |\langle U(t,0)( \psi_{z_1}) U(t,0)(
	\psi_{z_2}), U(t,0) ( \psi_{z_3}) U(t,0) ( \psi_{z_4})
	\rangle_{L^2_{t,x}( \eta(t) dx dt)}|.
	\end{equation}
	There is no difficulty with interchanging the order of integration
	as $f$ was assumed to be Schwartz. We claim
	\begin{prop}
		\label{p:L2_kernel_bound}
		For some $0 < \theta < 1$
		the kernel
		\[
		K(z_1, z_2, z_3, z_4) \max (\langle z_1 - z_2 \rangle^\theta,  \langle
		z_3 - z_4 \rangle^\theta)
		\]
		is bounded as a map on $L^2(T^* \mf{R} \times T^* \mf{R})$.
	\end{prop}
	Let us first see how this proposition
	implies the previous one. Writing $a_z = | \langle f, \psi_z
	\rangle|$, we have
	\[
	\begin{split}
	\|U(t,0) f\|_{L^4_{t, x}(\eta(t) dx dt)}^4 &\lesssim \Bigl( \int_{(T^* \mf{R})^2}
	a_{z_1}^2 a_{z_2}^2 \langle z_1 - z_2 \rangle^{-2\theta} \, dz_1
	dz_2 \Bigr)^{1/2} \Bigl( \int_{ (T^* \mf{R})^2} a_{z_3}^2 a_{z_4}^2
	\, dz_3 dz_4
	\Bigr)^{1/2}\\
	&\lesssim \|f\|_{L^2}^2 \Bigl( \int_{(T^* \mf{R})^2} a_{z_1}^2
	a_{z_2}^2\langle z_1 - z_2 \rangle^{-2\theta} \, dz_1 dz_2
	\Bigr)^{1/2}.
	\end{split}
	\]
	By Young's inequality, the convolution kernel $k(z_1, z_2) = \langle z_1 - 
	z_2 \rangle^{-2\theta}$ is bounded from $L^p_z$ to $L^{p'}_{z}$ for
	some $p \in (1, 2) $, and the integral on the right is bounded by
	\[
	\Bigl(\int_{T^* \mf{R} } a_{z}^{2p} \, dz
	\Bigr)^{2/p} \le \| f\|_{L^2}^{4/p} \sup_{z} a_z^{4/p'}.
	\]
	This yields
	\[
	\|U(t,0) f\|_{L^4_{t, x}(\eta(t) dx dt)} \lesssim \| f\|_{L^2}^{\fr{1}{2} + \fr{1}{2p}}
	\sup_z a_z^{\fr{1}{2p'}},
	\]
	which settles Proposition~\ref{P:4.1} with $\beta=\fr{1}{2p'}$.
	
	It remains to prove Proposition~\ref{p:L2_kernel_bound}. 
	Lemma~\ref{l:galilei} implies that $	U(t, 0)(\psi_{z_j})(x) = 
	e^{i\alpha_j} [U_j(t, 0) \psi](x - x_j^t)$, 
	where
	\[
	\alpha_j(t, x) = (x - x_j^t) \xi_j^t + \int_0^t \fr{1}{2}
	|\xi_j^\tau|^2 - V(\tau, x_j^\tau) \, d\tau
	\]
	and $U_j$ is the propagator for $H_j = -\tfr{1}{2}\partial_x^2 + V_j(t,
	x)$ with
	\begin{equation}
	\label{e:Vj}
	V_j(t, x) = x^2 \int_0^1 (1-s)\partial^2_x V (t, x_j^t + sx) \, ds.
	\end{equation}
	The envelopes $[U_j(t, 0) \psi] ( x - x_j^t)$ concentrate along
	the classical trajectories $t \mapsto x_j^t$:
	\begin{equation}
	\label{e:wavepacket_schwartz_tails}
	|\partial_x^k [U_j(t, 0) \psi] (x - x_j^t)| \lesssim_{k, N}
	\langle x- x_j^t \rangle^{-N}.
	\end{equation}
	
		The kernel $K$ therefore admits the crude bound
	\[
	K(\vec{z}) \lesssim_N \int \prod_{j=1}^4 \langle x - x_j^t
	\rangle^{-N}  \,  \eta(t) dx dt \lesssim \max ( \langle z_1 - z_2
	\rangle, \langle z_3 - z_4 \rangle)^{-1},
	\]
	and Proposition~\ref{p:L2_kernel_bound} will follow from
	\begin{prop}
		\label{p:L2_kernel_bound_delta}
		For $\delta > 0$ sufficiently small, the operator with kernel 
		$K^{1-\delta}$ is
		bounded on $L^2(T^* \mf{R} \times T^* \mf{R})$.
	\end{prop}
	\begin{proof}
		We partition the 4-particle phase space $(T^* \mf{R})^4$ according to 
		the degree of interaction between the particles. Define
		\[
		\begin{split}
		E_0 &=\{ \vec{z} \in (T^* \mf{R})^4 : \min_{|t| \le \delta_0} \max_{k, 
			k'} |x_k^t - x_{k'}^t|
		\le 1\}\\
		E_m &= \{ \vec{z} \in (T^* \mf{R})^4 : 2^{m-1} < \min_{|t| \le 
			\delta_0} \max_{k, k'} |x_k^t - x_{k'}^t|
		\le 2^m \}, \ m \ge 1,
		\end{split}
		\]
		and decompose
		\[
		K = K \mr{1}_{E_0} + \sum_{m \ge 1} K \mr{1}_{E_m} = K_0 + \sum_{m \ge 
			1} K_m.
		\]
		Then
		\[
		K^{1-\delta} = K_0^{1-\delta} + \sum_{m \ge 1} K_m^{1-\delta}.
		\]
		
		The $K_0$ term heuristically corresponds to the 4-tuples of 
		wavepackets 
		that all
		collide at some time $t \in [-\delta_0, \delta_0]$ and will be the 
		dominant term thanks to the 
		decay 
		in~\eqref{e:wavepacket_schwartz_tails}.  We will show that for any $m\geq 0$ and any $N > 0$,
		\begin{equation}
		\label{e:Km_L2_bound}
		\|K_m^{1-\delta} \|_{L^2 \to L^2} \lesssim_N 2^{-mN},
		\end{equation}
		which immediately implies the proposition upon summing. In turn, 
		this 
		will be a consequence of the following pointwise bound: 
		
		\begin{lma}
			\label{l:K0_ptwise}
			For each $m\geq 0$ and $\vec{z} \in E_m$, let $t(\vec{z})$ be a time 
			witnessing the 
			minimum in the definition of $E_m$. Then for any $N_1, N_2 \ge 0$,
			\begin{equation}
			\nonumber
			\begin{split}
			&| K_m(\vec{z}) | \lesssim_{N_1,N_2} 2^{-mN_1} \min \biggl[\fr{|
				\xi^{t(\vec{z})}_1 + \xi^{t(\vec{z})}_2 - \xi^{t(\vec{z})}_3 - 
				\xi^{t(\vec{z})}_4|^{-N_2}}{1 + |\xi^{t(\vec{z})}_1 - 
				\xi^{t(\vec{z})}_2| +
				|\xi^{t(\vec{z})}_3 - \xi^{t(\vec{z})}_4|}, \fr{ 1 + 
				|\xi^{t(\vec{z})}_1 - \xi^{t(\vec{z})}_2| + |\xi^{t(\vec{z})}_3 
				- \xi^{t(\vec{z})}_4|}{ \bigl | (\xi^{t(\vec{z})}_1 -
				\xi^{t(\vec{z})}_2)^2 - (\xi^{t(\vec{z})}_3 - 
				\xi^{t(\vec{z})}_4)^2 \bigr|^2} \biggr].
			\end{split}
			\end{equation}
		\end{lma}
		Deferring the proof for the 
		moment, let us see how Lemma~\ref{l:K0_ptwise} implies \eqref{e:Km_L2_bound}.
		By Schur's test and symmetry, it suffices to show that
		\begin{align}
		\label{e:Km_schur}
		\sup_{z_3,z_4}\int K_m(z_1, z_2, z_3, z_4)^{1-\delta} \, dz_1 dz_2\lesssim_N 2^{-mN},
		\end{align}
		where the supremum is taken over all $(z_3, z_4)$ in the image of the projection
		$E_m \subset (T^* \mf{R})^4 \to T^*\mf{R}_{z_3} \times
		T^*\mf{R}_{z_4}$.  Fix such a pair $(z_3,z_4)$ and let
		\[
		E_m(z_3, z_4) = \{ (z_1, z_2) \in (T^*\mf{R})^2 : (z_1, z_2, z_3, z_4) 
		\in E_m\}.
		\]
		Choose $t_1 \in [-\delta_0, \delta_0]$
		minimizing $|x_3^{t_1} - x_4^{t_1}|$; the definition of $E_m$ implies 
		that
		$|x_3^{t_1} - x_4^{t_1}| \le 2^{m}$.
		
		Suppose $(z_1, z_2) \in E_m(z_3, z_4)$. By 
		Lemma~\ref{l:collisions}, any 
		``collision time" $t(z_1, z_2, z_3, z_4)$ must belong to the interval
		\[
		I = \bigl\{ t \in [-\delta_0, \delta_0] : |t - t_1| \lesssim \min\Bigl 
		(1,
		\fr{2^m}{|\xi_3^{t_1}
			- \xi_4^{t_1} |}\Bigr ) \bigr\},
		\]
		and for such $t$ one has
		\[
		\begin{split}
		|x_3^t - x_4^t| \lesssim 2^{m}, \quad		| \xi_3^t - \xi_4^t - 
		(\xi_3^{t_1} - \xi_4^{t_1})| \lesssim
		\min \Bigl ( 2^m, \fr{2^{2m}}{ |\xi_3^{t_1} - \xi_4^{t_1}|} \Bigr ).
		\end{split}
		\]		
		The contribution of each $(z_1, z_2) \in E_m(z_3, z_4)$ to the 
		integral~\eqref{e:Km_schur} will depend on their relative momenta 
		at the collision time. We organize the integration domain 
		$E_m(z_1, z_2)$ accordingly.
		
		Write $Q_{\xi} = (0, \xi) + [-1, 1]^2 \subset T^* \mf{R}$, and denote
		by $\Phi(t, s)$ the classical propagator for the Hamiltonian
		\[
		h = \tfrac{1}{2}|\xi|^2 + V(t, x).
		\]
		Using the shorthand $z^t = \Phi(t, 0)(z)$, for $\mu_1, \mu_2\in \R$ we define
		\[
		Z_{\mu_1, \mu_2} = \bigcup_{t \in I} (\Phi(t, 0) \otimes \Phi(t, 
		0))^{-1}
		\Bigl( \fr{z_3^t + z_4^t}{2} + 2^{m} Q_{\mu_1} \Bigr) \times \Bigl(
		\fr{z_3^t + z_4^t}{2} + 2^{m} Q_{\mu_2} \Bigr),
		\]
		where $\Phi(t, 0) \otimes \Phi(t, 0)(z_1, z_2) = (z_1^t, z_2^t)$ is
		the product flow on $T^* \mf{R} \times T^* \mf{R}$. This set is 
		depicted schematically in Figure~\ref{f:fig} when $m = 0$. This
		corresponds to 
		the pairs of wavepackets $(z_1, z_2)$ with momenta $(\mu_1, \mu_2)$ 
		relative to 
		the wavepackets $(z_3, z_4)$, when all four wavepackets interact. We have
		\[
		E_m(z_3, z_4) \subset \bigcup_{\mu_1, \mu_2 \in \mf{Z}} Z_{\mu_1, 
			\mu_2}.
		\]
		
		\begin{figure}
			\includegraphics[scale=1]{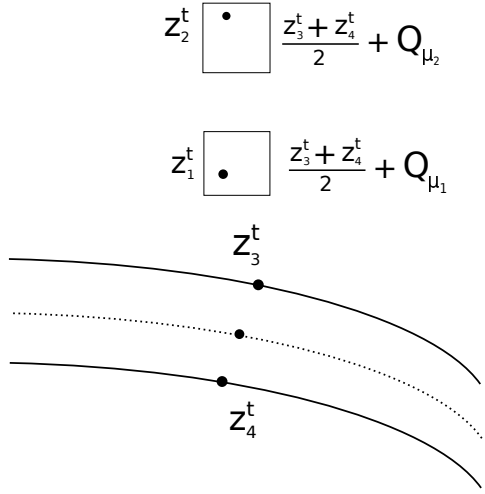}
			\caption{$Z_{\mu_1, \mu_2}$ comprises all $(z_1, z_2)$ such that 
			$z_1^t$ and $z_2^t$ belong to the depicted phase 
			space boxes for $t$ in the interval $\textrm{I}$.}
					\label{f:fig}
		\end{figure}
		\begin{lma}
			\label{l:measure}
			$|Z_{\mu_1, \mu_2}| \lesssim 2^{4m} \max(1,|\mu_1|, |\mu_2|) |I|$, 
			where
			$| \cdot |$ on the left denotes Lebesgue measure on $(T^* 
			\mf{R})^2$.
		\end{lma}
		
		\begin{proof}
			Without loss assume $|\mu_1| \ge |\mu_2|$. Partition the interval
			$I$ into
			subintervals of length $|\mu_1|^{-1}$ if $\mu_1\neq 0$ and in subintervals of length $1$ if $\mu_1=0$. For each $t'$ in the
			partition, Lemma~\ref{l:technical_lma} implies that for some 
			constant $C > 0$ we have
			\[
			\begin{split}
			\bigcup_{|t-t'| \le \min(1,|\mu_1|^{-1})} \Phi(t, 0)^{-1} \Bigl(
			\fr{z_3^t+z_4^t}{2} + 2^{m}Q_{\mu_1} \Bigr) &\subset \Phi(t',
			0)^{-1} \Bigl ( \fr{z_3^{t'} + z_4^{t'}}{2} + C2^{m} Q_{\mu_1}
			\Bigr)\\
			\bigcup_{|t-t'| \le \min(1,|\mu_1|^{-1})} \Phi(t, 0)^{-1} \Bigl(
			\fr{z_3^t+z_4^t}{2} + 2^{m}Q_{\mu_2} \Bigr) &\subset \Phi(t',
			0)^{-1} \Bigl ( \fr{z_3^{t'} + z_4^{t'}}{2} + C2^{m} Q_{\mu_2}
			\Bigr),
			\end{split}
			\]
			and so
			\[
			\begin{split}
			&\bigcup_{|t-t'| \le \min(1,|\mu_1|^{-1})} (\Phi(t, 0) \otimes \Phi(t, 
			0))^{-1}
			\Bigl( \fr{z_3^t + z_4^t}{2} + 2^{m} Q_{\mu_1} \Bigr) \times \Bigl(
			\fr{z_3^t + z_4^t}{2} + 2^{m} Q_{\mu_2} \Bigr)\\
			&\subset (\Phi(t', 0) \otimes \Phi(t', 0))^{-1} \Bigl ( 
			\fr{z_3^{t'} + z_4^{t'}}{2} + C2^{m} Q_{\mu_1}
			\Bigr) \times  \Bigl ( \fr{z_3^{t'} + z_4^{t'}}{2} + C2^{m} 
			Q_{\mu_2}
			\Bigr).
			\end{split}
			\]
			By Liouville's theorem, the right side has measure $O(2^{4m})$ in
			$(T^* \mf{R})^2$. The claim follows by summing over the partition.
		\end{proof}
		
		For each $(z_1, z_2) \in E_m(z_3, z_4) \cap Z_{\mu_1, \mu_2}$, suppose
		that $z_j^t \in \tfr{z_3^t + z_4^t}{2} + 2^m
		Q_{\mu_j}$ for some $t \in I$. As
		\[
		\xi_j^{t} = \fr{\xi_3^t + \xi_4^t}{2} + \mu_j + O(2^m), \ j = 1, 2,
		\]
		the second assertion of Lemma~\ref{l:collisions} implies that
		\begin{gather*}
		\xi_1^{t(\vec{z})} + \xi_2^{t(\vec{z})} - \xi_3^{t(\vec{z})} -
		\xi_4^{t(\vec{z})} = \mu_1 + \mu_2 + O(2^m)\\
		\xi_1^{t(\vec{z})} - \xi_2^{ t(\vec{z})} = \mu_1 - \mu_2 + O(2^m).
		\end{gather*}
		Hence by Lemma~\ref{l:K0_ptwise},
		\begin{equation}
		\nonumber
		\begin{split}
		| K_m | &\lesssim_{N} 2^{-3mN} \min \biggl[\fr{\langle
			\mu_1 + \mu_2 + O(2^m)\rangle^{-N}}{1 + \bigl | |\mu_1 - \mu_2| +
			|\xi^{t_1}_3 - \xi^{t_1}_4| + O(2^m) \bigr |}, \fr{ 1 + |\mu_1 - 
			\mu_2| +
			|\xi^{t_1}_3 - \xi^{t_1}_4| + O(2^m)}{ \bigl | (\mu_1 - \mu_2)^2 -
			(\xi^{t_1}_3 - \xi^{t_1}_4)^2  + O(2^{2m}) \bigr|^2} \biggr]\\
		&\lesssim_N 2^{(5-2N)m} \min\biggl[ \fr{ \langle \mu_1 + \mu_2
			\rangle^{-N}} {1 + |\mu_1 - \mu_2| + |\xi_3^{t_1} - \xi_4^{t_1}| },
		\fr{ 1 + |\mu_1 - \mu_2| + |\xi_3^{t_1} - \xi_4^{t_1}|}{ \bigl | (\mu_1 
			-
			\mu_2)^2 - (\xi_3^{t_1} - \xi_4^{t_1})^2 \bigr |^2} \biggr].
		\end{split}
		\end{equation}
		
				Applying Lemma~\ref{l:measure}, writing $\max(|\mu_1|, |\mu_2|) \le
		|\mu_1 + \mu_2| + |\mu_1 - \mu_2|$, and absorbing $|\mu_1 + \mu_2|$
		into the factor $\langle \mu_1 + \mu_2 \rangle^{-N}$ by adjusting $\delta$, 
		\[
		\begin{split}
		&\int K_m(z_1, z_2, z_3, z_4)^{1-\delta} \, dz_1 dz_2 \le
		\sum_{\mu_1, \mu_2 \in \mf{Z}} \int K_m(z_1, z_2, z_3, z_4)^{1-\delta}
		\mr{1}_{ Z_{\mu_1, \mu_2}}(z_1, z_2) \, dz_1 dz_2\\
		&\lesssim \sum_{\mu_1, \mu_2 \in \mf{Z}} 2^{-mN} \min\Bigl( \fr{ 
			\langle \mu_1 + \mu_2
			\rangle^{-N}} {1 + |\mu_1 - \mu_2| + |\xi_3^{t_1} - \xi_4^{t_1}| },
		\fr{ 1 + |\mu_1 - \mu_2| + |\xi_3^{t_1} - \xi_4^{t_1}|}{ \bigl | (\mu_1 
			-
			\mu_2)^2 - (\xi_3^{t_1} - \xi_4^{t_1})^2 \bigr |^2} \Bigr 
			)^{1-\delta}  \fr{1 + 
			|\mu_1 -
			\mu_2|}{1 + |\xi_3^{t_1} - \xi_4^{t_1}|}.
		\end{split}
		\]
		When $|\mu_1 - \mu_2| \le 1$, we choose the first term in the minimum 
		to see that the sum is of size $2^{-mN}$. If $|\mu_1-\mu_2|\geq\max(1, 2|\xi_3^{t_1} - \xi_4^{t_1}|)$, the above
		expression is bounded by
		\[
		\sum_{\mu_1, \mu_2 \in \mf{Z}} 2^{-mN} \min \Bigl ( \frac{\langle \mu_1 +
		\mu_2 \rangle^{-N}}{\langle\mu_1-\mu_2\rangle}, \fr{1}{\langle\mu_1-\mu_2\rangle^3}
		\Bigr )^{1-\delta} \langle\mu_1-\mu_2\rangle\lesssim_N 2^{-mN}.
		\]
	If instead $1\leq |\mu_1-\mu_2|\leq 2|\xi_3^{t_1} - \xi_4^{t_1}|$, we obtain the bound
	\[
		\sum_{\mu_1, \mu_2 \in \mf{Z}} 2^{-mN} \min \Bigl ( \frac{\langle \mu_1 +
		\mu_2 \rangle^{-N}}{\langle|\xi_3^{t_1} - \xi_4^{t_1}|\rangle}, \fr{1}{\bigl[|\mu_1-\mu_2|-|\xi_3^{t_1} - \xi_4^{t_1}|\bigr]^2}
		\Bigr )^{1-\delta} \lesssim_N 2^{-mN}.
		\]
		Therefore
		\[
		\int K_m (z_1, z_2, z_3, z_4)^{1-\delta} dz_1 dz_2 \lesssim_N 2^{-mN},
		\]
		which gives \eqref{e:Km_schur}. Modulo Lemma~\ref{l:K0_ptwise}, this completes 
		the
		proof of Proposition~\ref{p:L2_kernel_bound_delta}.
	\end{proof}

	\section{Proof of Lemma~\ref{l:K0_ptwise}}
	The spatial localization and 
	the definition of $E_m$
	immediately imply the cheap bound
	\[
	|K_m(\vec{z})| \lesssim_N 2^{-mN}.
	\]
	However, we can often do better by exploiting oscillation in 
	space and time. 
	As the argument
	is essentially the same for all $m$, we shall for simplicity take
	$m = 0$ in the sequel. We shall also assume that $t(\vec{z}) = 0$ as the 
	general case involves little more 
	than replacing all instances of $\xi$ in the sequel by $\xi^{t(\vec{z})}$.
	
	By Lemma~\ref{l:galilei},
	\begin{equation}
	\nonumber
	K_0(\vec{z}) = \Bigl | \iint e^{i\Phi} \prod_{j=1}^4 U_j(t,
	0) \psi (x - x_j^t) \, \eta(t) dx dt \Bigr |,
	\end{equation}
	where
	\[
	\Phi  = \sum_j \sigma_j \Bigl[ (x - x_j^t)\xi_j^t + \int_0^t
	\fr{1}{2}|\xi_j^\tau|^2 - V(\tau, x_j^\tau) \, d\tau \Bigr]
	\]
	 with $\sigma = (+, +, -, -)$ and
	 $\prod_{j=1}^4 c_j := c_1 c_2 \overline{c}_3 \overline{c}_4$. To save 
	 space we abbreviate $U_j(t, 0)$ as $U_j$.
	
	Let $1 =\sum_{\ell \ge 0} \theta_{\ell}$ be a partition 
	of
	unity such that $\theta_0$ is supported in the unit ball and
	$\theta_{\ell}$ is supported in the annulus
	$\{ 2^{\ell-1} < |x| < 2^{\ell + 1} \}$. Also choose
	$\chi \in C^\infty_0$ equal to $1$ on $|x| \le 8$.  Further bound
	$K_0 \le \sum_{\vec{\ell}} K_0^{\vec{\ell}}$, where
	\begin{equation}
	\nonumber
	K_0^{\vec{\ell}}(\vec z) = \Bigl | \iint e^{i\Phi} \prod_{j} U_j \psi(x -
	x_j^t) \theta_{\ell_j}(x - x_j^t) \, \eta(t)
	dx dt \Bigr |.
	\end{equation}
	
	Fix $\vec{\ell}$ and write $\ell^* = \max \ell_j$. By 
	Lemma~\ref{l:collisions}, the integrand is
	nonzero only in the spacetime region
	\begin{equation}
	\label{e:support}
	\{(t, x) : |t|
	\lesssim \min(1, \tfr{2^{\ell^*} }{ \max |\xi_j - \xi_k|})\quad\text{and}\quad  |x-x_j^t| 
	\lesssim 2^{\ell_j} \},
	\end{equation}
	and for all $t$ subject to the above restriction we have
	\begin{equation}
	\label{e:xi_constancy}
	|x_j^t - x_k^t | \lesssim 2^{\ell^*}\quad\text{and}\quad |\xi_j^t - \xi_k^t - (\xi_j-\xi_k)| \lesssim \min\bigl( 
	2^{\ell^*},	\tfrac{2^{2\ell^*}}{\max |\xi_j - \xi_k|}\bigr).
	\end{equation}

	We estimate $K_0^{\vec{\ell}}$ using integration by parts. The
	relevant derivatives of the phase function are
	\begin{equation}
	\partial_x \Phi = \sum_j \sigma_j \xi_j^t,	\qquad \partial^2_x
	\Phi = 0,	\qquad -\partial_t \Phi = \sum_j \sigma_j h(t, z_j^t) + \sum_j \sigma_j (x
	- x_j^t) \partial_x V (t, x_j^t).
		\end{equation}
	Integrating by parts repeatedly in $x$ and using \eqref{e:wavepacket_schwartz_tails}, for any $N \ge 0$, we get
	\begin{equation}
	\label{e:x-int_by_parts}
	\begin{split}
	|K_0^{\vec{\ell}}(\vec z)| &\lesssim_N \int |\xi_1^t +
	\xi_2^t - \xi_3^t - \xi_4^t|^{-N}
	\bigl | \partial_x^N \prod  U_j \psi (x - x_j^t) \theta_{\ell_j}
	(x - x_j^t) \bigr | \, \eta(t) dx dt \\
	&\lesssim_N \fr{2^{-\ell^* N} \langle \xi_1 + \xi_2 - \xi_3 - \xi_4
		\rangle^{-N} }{1 + |\xi_1 - \xi_2| + |\xi_3 - \xi_4|},
	\end{split}
	\end{equation}
	where we have used~\eqref{e:xi_constancy} to
	replace $ \xi_1^t + \xi_2^t - \xi_3^t
	- \xi_4^t$ with $\xi_1 + \xi_2 -
	\xi_3 - \xi_4 + O(2^{\ell^*})$.
	
	We can also exhibit gains from oscillation in time. Naively, one 
	might integrate by parts using the differential operator $\partial_t$, but 
	better decay 
	can be 
	obtained by accounting for 
	the bulk 
	motion of the wavepackets in addition to the phase. If one pretends that 
	the envelope $U_j \psi( x - x_j^t) ``\approx" \phi(x - 
	x_j^t)$ is simply transported along the classical trajectory, then 
	\begin{align*}
	(\partial_t + \xi_j^t \partial_x)U_j \psi(x - x_j^j) ``\approx" (-\xi_j^t 
	+ 
	\xi_j^t) \phi'(x-x_j^t) = 0.
	\end{align*}
	Motivated by this heuristic, we introduce a vector field adapted 
	to the 
	average bicharacteristic for the four wavepackets. This will have the 
	greatest effect when the wavepackets all follow nearby bicharacteristics; 
	when 
	they are far apart in phase space, 
	we can
	exploit the strong spatial localization and the fact 
	that 
	two 
	wavepackets 
	widely separated in momentum will interact only for a short time.
	
	Define 
	\begin{gather*}
	\overline{x^t} = \tfrac{1}{4} \sum_j x_j^t, \qquad \overline{\xi^t} = \tfrac{1}{4}\sum_j \xi_j^t, \qquad	x_j^t = \overline{x^t} + \overline{x^t}_j,  \qquad \xi_j^t = \overline{\xi^t} 
	+ \overline{\xi^t}_j.
	\end{gather*}
	The variables $(\overline{x^t}_j, \overline{\xi^t}_j)$ describe the 
	location of
	the $j$th wavepacket in phase space relative to the average 
	$(\overline{x^t}, 
	\overline{\xi^t})$; see Figure~\ref{f:fig2}.
	\begin{figure}
		\includegraphics[scale=0.7]{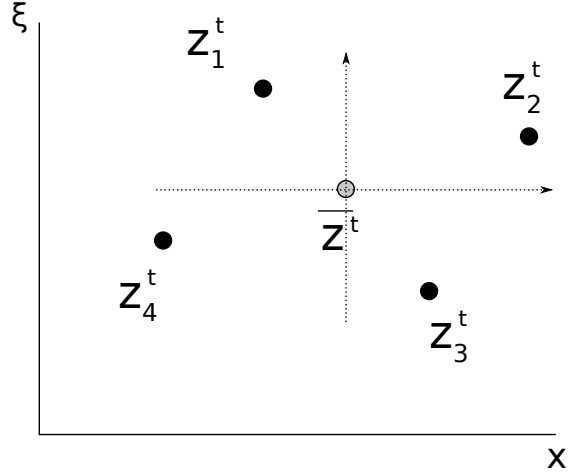}
		\caption{Phase space coordinates relative to the ``center of mass".}
				\label{f:fig2}
	\end{figure}
	We have
	\begin{equation}
	\label{e:cm_derivs}
	\begin{split}
	\frac{d}{dt} \overline{x^t}_j &= \overline{\xi^t}_j = O( \max_{j, k} 
	|\xi_j^t - 
	\xi_k^t|) = O\Bigl( |\xi_j - \xi_k| + \min\bigl(2^{\ell^*}, \tfrac{ 2^{2\ell^*}}{ 
		\max |\xi_j - \xi_k|} \bigr)\Bigr) \\
	\frac{d}{dt} 
	\overline{\xi^t}_j &= \tfrac{1}{4} \sum_k \partial_x V(t, x^t_k) - 
	\partial_x 
	V(t, x^t_j)\\
	&= \tfrac{1}{4} \sum_k  (x^t_k - x^t_j) \int_0^1 \partial^2_x V(t, 
	(1-\theta) x_j^t + \theta x_k^t) \, d\theta\\
	&= O(2^{\ell^*}).
	\end{split}
	\end{equation}
	Note that 
	\begin{equation}
	\label{e:cm_var_max}
	\max_j |\overline{x^t}_j| \sim \max_{j, k} |x^t_j - x^t_k|, \quad \max_{j}
	| \overline{\xi^t}_j | \sim \max_{j, k} | \xi^t_j - \xi^t_k|.
	\end{equation}
	
	Consider the operator
	\[
	D = \partial_t + \overline{\xi^t} \partial_x.
	\]
	We compute
	\begin{equation}
	\nonumber
	\begin{split}
	-D\Phi &= \sum \sigma_j h(t, z_j^t) + \sum \sigma_j [(x -
	x_j^t) \partial_x V(t, x_j^t) - \overline{\xi^t} \xi_j^t ]\\
	&= \tfrac{1}{2}\sum \sigma_j |\overline{\xi^t}_j|^2 + \sum \sigma_j \bigl[
	V(t, x_j^t) + (x-x_j^t) \partial_x V (t, x_j^t)\bigr].
	\end{split}
	\end{equation}
	This is more transparent when expressed in the relative
	variables $\overline{x^t}_j$ and $\overline{\xi^t}_j$.  Each term in the
	second sum can be written as
	\begin{align*}
	V(t, \overline{x^t} &+ \overline{x^t}_j) + (x - x_j^t) \partial_x V(t,\overline{x^t} + \overline{x^t}_j)\\
	&= V(t, \overline{x^t} + \overline{x^t}_j) - V(t, \overline{x^t}) -\overline{x^t}_j \partial_x V(t, \overline{x^t})
	+ V(t, \overline{x^t}) + \overline{x^t}_j \partial_x V (t, \overline{x^t})+ (x-x_j^t) \partial_x V(t, \overline{x^t})\\
	&\quad+ (x-x_j^t) ( \partial_x V(t, \overline{x^t} + \overline{x^t}_j)- \partial_x V(t, \overline{x^t}) )\\
	&= V^{\overline{z}} (t, \overline{x^t}_j) + V(t, \overline{x^t}) +(x-x_j^t)  [\partial_x V^{\overline{z}}](t,\overline{x^t}_j) +(x-\overline{x^t}) \partial_x V(t, \overline{x^t}),
	\end{align*}
	where
	\begin{equation}
	\label{e:V_cm}
	V^{\overline{z}}(t, x) = V(t, \overline{x^t} + x) - V(t, \overline{x^t}) -
	x \partial_x V(t, \overline{x^t}) = x^2 \int_0^1
	(1-s)\partial^2_x V(t, \overline{x^t} + s x) \, ds.
	\end{equation}
	The terms without the subscript $j$ cancel upon summing, and we obtain
	\begin{equation}
	\label{e:vf_deriv}
	-D\Phi = \tfrac{1}{2} \sum \sigma_j |\overline{\xi^t}_j|^2 + \sum
	\sigma_j [V^{\bar{z}}(t, \overline{x^t}_j) + (x-x_j^t) [\partial_x
	V^{\overline{z}}](t, \overline{x^t}_j)].
	\end{equation}
	Therefore the contribution to~$D\Phi$ from $V$ depends
	essentially only on the differences $x^t_j - x^t_k$. Invoking
	\eqref{e:support}, \eqref{e:xi_constancy}, 
	and~\eqref{e:cm_var_max}, we see that 
	the second sum is at most $O(2^{2\ell^*})$.
	
	Note also that
	\[
	(\overline{\xi^t}_j)^2 = (\overline{\xi}_j)^2 + O(2^{2\ell^*}),
	\]
	as can be seen via \eqref{e:cm_derivs},
	the fundamental theorem of calculus, and the time restriction 
	\eqref{e:support}. It follows that if
	\begin{equation}
	\label{e:energy_cond}
	\Bigl | \sum_j \sigma_j ( \overline{\xi}_j)^2 \Bigr | \ge
	C \cdot 2^{2\ell^*}
	\end{equation}
	for some large constant $C > 0$, then on the support of
	the integrand
	\begin{equation}
	\label{e:phasederiv_lower_bound}
	\begin{split}
	|D \Phi| \gtrsim \Bigl | \sum_{j} \sigma_j
	(\overline{\xi}_j)^2 \Bigr | &= \tfrac{1}{2}\bigl | (\overline{\xi}_1 +
	\overline{\xi}_2)^2 - (\overline{\xi}_3 + \overline{\xi}_4)^2 +
	(\overline{\xi}_1 - \overline{\xi}_2)^2 - (\overline{\xi}_3 -
	\overline{\xi}_4)^2\bigr| \\
	&=\tfrac12 \bigl | |\xi_1 - \xi_2|^2 - |\xi_3 -
	\xi_4|^2 \bigr |,
	\end{split}
	\end{equation}
	where the last inequality follows from the fact that
	$\overline{\xi}_1 + \overline{\xi}_2 + \overline{\xi}_3 +
	\overline{\xi}_4 = 0$.
	
	The second derivative of the phase is
	\begin{equation}
	\nonumber
	\begin{split}
	-D^2 \Phi
	&= \sum \sigma_j \overline{\xi^t}_j \bigl[\tfrac{1}{4}\sum_k \partial_x V (t, x_k^t) - \partial_x V (t, x_j^t) \bigr]+ \sum_j \sigma_j (x -x_j^t) \xi_j^t \partial^2_x V(t, x_j^t) \\
	&\quad+ \overline{\xi^t} \sum	\sigma_j \partial_x V (t, x_j^t)+ \sum \sigma_j [\partial_t V(t, x_j^t) +(x-x_j^t) \partial_t \partial_x V(t, x_j^t)]\\
	&= \sum  \sigma_j \overline{\xi^t}_j \bigl[\tfrac{1}{4}\sum_k \partial_x V(t, x_k^t) - \partial_x V(t, x_j^t) \bigr]+ \sum \sigma_j	(x - x_j^t) \overline{\xi^t}_j \partial_x^2 V(t, x_j^t)\\
	&\quad+ \sum \sigma_j [\partial_t V (t, x_j^t) +	(x-x_j^t) \partial_t \partial_x V (t, x_j^t)] +	\overline{\xi^t} \sum \sigma_j[ \partial_x V(t, x_j^t) +(x-x_j^t) \partial^2_x V (t, x_j^t)].
	\end{split}
	\end{equation}
	We rewrite the last two sums as before to obtain
	\begin{equation}
	\label{e:phase_second_deriv}
	\begin{split}
	-D^2 \Phi &= \sum \sigma_j \overline{\xi^t}_j \bigl[\tfrac{1}{4}
	\sum_k \partial_x V(t, x_k^t) - \partial_x V(t, x_j^t) \bigr]+ \sum \sigma_j
	(x - x_j^t) \overline{\xi^t}_j \partial_x^2 V(t, x_j^t) \\
	&\quad+ \sum \sigma_j [(\partial_t V)^{\overline{z}} (t, \overline{x^t}_j) +
	(x-x_j^t) \partial_x (\partial_t V)^{\overline{z}} (t,
	\overline{x^t}_j)]\\
	&\quad+ \overline{\xi^t} \sum \sigma_j [(\partial_x V)^{\overline{z}} (t, 
	\overline{x^t}_j) +
	(x-x_j^t) \partial_x (\partial_x V)^{\overline{z}} (t,
	\overline{x^t}_j)],
	\end{split}
	\end{equation}
	where
	\[
	\begin{split}
	(\partial_t V)^{\overline{z}}(t, x) &= x^2 \int_0^1
	(1-s)\partial^2_x \partial_t V(t, \overline{x^t} + s x) \, ds\\
	(\partial_x V)^{\overline{z}}(t, x) &= x^2 \int_0^1
	(1-s)\partial^3_x V(t, \overline{x^t} + s x) \, ds.
	\end{split}
	\]

	

	Assume that~\eqref{e:energy_cond} holds. Write $e^{i\Phi} = 
	\frac{D\Phi}{i|D\Phi|^2} \cdot D e^{i\Phi}$ and integrate by parts to get
	\[
	\begin{split}
	K^{\vec{\ell}}_0(\vec z) &\lesssim \Bigl | \int e^{i\Phi} \frac{ D^2 \Phi}{ (D \Phi)^2} 
	\prod U_j \psi(x - x_j^t) \theta_{\ell_j} (x-x_j^t) \, \eta(t) dx dt \Bigr| 
	\\
	&+ \Bigl| \int e^{i\Phi} \frac{1}{(D\Phi)} D \prod  U_j \psi(x - x_j^t) 
	\theta_{\ell_j} (x-x_j^t) \, \eta(t) dx dt \Bigr|\\
	&\lesssim \Bigl | \int e^{i\Phi} \fr{ D^2 \Phi}{(D
		\Phi)^2}  \prod U_j \psi (x - x_j^t) \theta_{\ell_j}
	(x - x_j^t)  \, \eta(t) dx dt \Bigr |\\
	&+ \Bigl | \int e^{i\Phi} \fr{2D^2 \Phi}{(D
		\Phi)^3} D \prod U_j \psi(x - x_j^t)
	\theta_{\ell_j} (x - x_j^t) \, \eta(t) dx dt\Bigr | \\
	&+  \Bigl | \int e^{i\Phi} \fr{1}{(D \Phi)^2} D^2\prod U_j \psi(x - x_j^t)
	\theta_{\ell_j} (x - x_j^t) \, \eta(t) dx dt\Bigr |\\
	&= I + II + III.
	\end{split}\]
	Note that after the first integration by parts, we only repeat the 
	procedure 
	for the second term. The point of this is to avoid higher derivatives of 
	$\Phi$, which may be unacceptably large due to factors of 
	$\overline{\xi^t}$.
	
	Consider first the contribution from $I$. Write $I \le I_a + I_b + I_c$, 
	where $I_a$,
	$I_b$, $I_c$ correspond respectively to the first, second, and
	third lines in the expression~\eqref{e:phase_second_deriv} for
	$D^2 \Phi$.
	
	In view of \eqref{e:wavepacket_schwartz_tails},
	\eqref{e:xi_constancy}, \eqref{e:cm_derivs}, and 
	\eqref{e:phasederiv_lower_bound},
	we have
	\[
	\begin{split}
	I_a &\lesssim_N \int \fr{ 2^{\ell^*} \sum_j |\overline{\xi^t}_j|}{
		|D \Phi|^2} \prod_j 2^{-\ell_j N} \chi \Bigl ( \fr{x -
		x_j^t}{2^{\ell_j}} \Bigr ) \, \eta(t) dx dt\\
	&\lesssim_N  \fr{ 2^{2\ell^*}(1 + \sum |\overline{\xi}_j|)}{ \bigl | (\xi_1 -
		\xi_2)^2 - (\xi_3 - \xi_4)^2 \bigr |^2} \cdot  \int \prod_j
	2^{-\ell_j N} \chi\Bigl (\fr{x - x_j^t}{2^{\ell_j}}\Bigr ) \,
	\eta(t) dx dt\\
	&\lesssim_N 2^{-\ell^*N} \cdot \fr{ \langle \xi_1 + \xi_2 - \xi_3 -
		\xi_4 \rangle + |\xi_1 - \xi_2| + |\xi_3 - \xi_4|}{ \bigl | (\xi_1 - 
		\xi_2)^2 -
		(\xi_3 - \xi_4)^2 \bigr |^2} \cdot  \fr{1}{1 + |\xi_1 -
		\xi_2| + |\xi_3 - \xi_4|},
	\end{split}
	\]
	where we have observed that 
	\[
	\begin{split}
	\sum_j |\overline{\xi}_j| &\sim \bigl (\sum_j
	|\overline{\xi}_j|^2\bigr)^{1/2} \sim \bigl( |\overline{\xi}_1 +
	\overline{\xi}_2|^2 + |\overline{\xi}_1 - \overline{\xi}_2|^2 +
	|\overline{\xi}_3 + \overline{\xi}_4|^2 + |\overline{\xi}_3 -
	\overline{\xi}_4|^2 \bigr)^{1/2}\\
	&\lesssim |\xi_1 + \xi_2 - \xi_3 - \xi_4|+ |\xi_1 - \xi_2| + |\xi_3 - 
	\xi_4|.
	\end{split}
	\]
	Similarly,
	\[
	\begin{split}
	I_b &\lesssim_N \int \fr{2^{2\ell^*}}{|D\Phi|^2} \prod_j
	2^{-\ell_j N} \chi \Bigl (\fr{x - x_j^t}{2^{\ell_j}}
	\Bigr ) \, \eta(t) dx dt\\
	&\lesssim_N \fr{2^{-\ell^* N}}{ \bigl | (\xi_1 - \xi_2)^2 -
		(\xi_3 - \xi_4)^2 \bigr |^2} \cdot  \fr{1}{1 + |\xi_1 -
		\xi_2| + |\xi_3 - \xi_4|}.
	\end{split}
	\]
	To estimate $I_c$, use the decay hypothesis
	$|\partial_x^3 V| \lesssim \langle x \rangle^{-1-\varepsilon}$ to
	obtain
	\[
	\begin{split}
	I_c &\lesssim_N \int \fr{2^{2\ell^*} |\overline{\xi^t}|}{|D
		\Phi|^2} \Bigl(\int_0^1 \sum_j \langle \overline{x^t} + 
	s\overline{x^t}_j
	\rangle^{-1-\varepsilon} \, ds\Bigr) \prod_j 2^{-\ell_j N} \chi
	\Bigl (\fr{x - x_j^t}{2^{\ell_j} } \Bigr ) \, \eta(t) dx
	dt\\
	&\lesssim_N \fr{2^{-\ell^* N}}{ |(\xi_1 - \xi_2)^2 - (\xi_3 -
		\xi_4)^2|^2}  \int_0^1\sum_j \int_{|t| \le \delta_0} |\overline{\xi^t}|
	\langle \overline{x^t} + s \overline{x^t}_j \rangle^{-1-\varepsilon}
	\, dt ds.
	\end{split}
	\]
	The integral on the right is estimated in the following technical lemma.
	\begin{lma}
		\label{l:technical_lma_2}
		\begin{equation}
		\nonumber
		\int_0^1\sum_j \int_{|t| \le \delta_0} |\overline{\xi^t}|
		\langle \overline{x^t} + s \overline{x^t}_j \rangle^{-1-\varepsilon}
		\, dt ds = O(2^{(2 + \varepsilon)\ell^*}).
		\end{equation}
	\end{lma}
	
	\begin{proof}
		It will be convenient to replace the average bicharacteristic
		$(\overline{x^t}, \overline{\xi^t})$ with the ray  $(\overline{x}^t, 
		\overline{\xi}^t)$ starting from the average initial data. We claim that
		\[
		|\overline{x^t} - \overline{x}^t| + | \overline{\xi^t} - 
		\overline{\xi}^t | = 
		O(2^{\ell^*})
		\] during the relevant
		$t$, for Hamilton's equations imply that
		\[
		\begin{split}
		\overline{x^t} - \overline{x}^t &= - \int_0^t (t-\tau) \Bigl ( \fr{1}{4}
		\sum_k \partial_x V(\tau, x_k^\tau) - \partial_x V (\tau,
		\overline{x}^\tau) \Bigr ) \, d\tau\\
		&= - \int_0^t (t-\tau) \Bigl ( \fr{1}{4} \sum_k ( \overline{x^\tau}_k +
		\overline{x^\tau} - \overline{x}^\tau) \int_0^1 \partial^2_x V
		(\tau, \overline{x}^\tau + s(x_k^\tau - \overline{x}^\tau)) \, ds
		\Bigr ) \, d\tau\\
		&= -\int_0^t(t-\tau)(\overline{x^\tau} - \overline{x}^\tau) 
		\Bigl[\int_0^1
		\fr{1}{4}\sum_k \partial^2_x V (\tau, \overline{x}^\tau + s(x_k^\tau -
		\overline{x}^\tau)) \, ds \Bigr] \, d\tau + O(2^{\ell^*}t^2),
		\end{split}
		\]
		and we can invoke Gronwall. Similar considerations yield the bound for 
		$|\overline{\xi^t} - 
		\overline{\xi}^t|$.
		As also $\overline{x^t}_j = O(2^{\ell^*})$, we are reduced to showing
		\begin{equation}
		\label{e:Ic_simplified_integral}
		\int_{|t| \le \delta_0} |\overline{\xi}^t| \langle
		\overline{x}^t \rangle^{-1-\varepsilon} \, dt = O(1).
		\end{equation}
		
		Integrating the ODE
		\[
		\tfr{d}{dt} \overline{x}^t = \overline{\xi}^t \quad \text{and}\quad \tfr{d}{dt}
		\overline{\xi}^t = - \partial_x V(t, \overline{x}^t)
		\]
		yields the estimates
		\[
		\begin{split}
		|\overline{x}^t - \overline{x}^s - (t-s) \overline{\xi}^s| &\le C 
		|t-s|^2(|\overline{x}^s| + |(t-s) \overline{\xi}^s|)\\
		|\overline{\xi}^t - \bar{\xi}^s| &\le C |t-s| (|\overline{x}^s| + 
		|(t-s)\overline{\xi}^s|)
		\end{split}
		\]
		for some constant $C$ depending on $ \| \partial_x^2 V\|_{L^\infty}$.
		By subdividing the time interval $[-\delta_0, \delta_0]$ if necessary, 
		we may 
		assume 
		in~\eqref{e:Ic_simplified_integral} that $(1+C)|t| \le 
		1/10$.
		
		Consider separately the cases
		$|\overline{x}| \le |\overline{\xi}|$ and
		$|\overline{x}| \ge |\overline{\xi}|$. When
		$|\overline{x}| \le |\overline{\xi}|$ we have
		\[
		2 |\overline{\xi}| \ge |\overline{\xi}^t| \ge |\overline{\xi}| - 
		\tfr{1}{5}|\overline{\xi}| \ge \tfr{1}{2} |\overline{\xi}|
		\]
		(assuming, as we may, that $|\overline{\xi}| \ge 1$) and the 
		bound~\eqref{e:Ic_simplified_integral} follows from the change of 
		variables 
		$y = \bar{x}^t$. If instead
		$|\overline{x}| \ge |\overline{\xi}|$, then
		$|\overline{x}^t| \ge \tfr{1}{2}|\overline{x}|$ and $ |\overline{\xi}^t|
		\le 2 |\overline{x}|$, which also yields the desired bound.
	\end{proof}
	
	Returning to $I_c$, we conclude that
	\[
	I_c \lesssim_N \fr{2^{-\ell^*N}}{ | (\xi_1 - \xi_2)^2 - (\xi_3 -
		\xi_4)^2 |^2}.
	\]
	Overall,
	\[
	\begin{split}
	I &\le I_a + I_b + I_c \lesssim_N 2^{-\ell^* N}  \fr{\langle \xi_1 + \xi_2 
		- \xi_3 -  \xi_4 \rangle}{ |(\xi_1 - \xi_2)^2 - (\xi_3 - \xi_4)^2|^2}.
	\end{split}
	\]
	
	For $II$, we have
	\begin{equation}
	\label{e:II_one_deriv}
	D[U_j \psi(x - x_j^t)] = -i H_j U_j  \psi(x - x_j^t) -
	\overline{\xi^t}_j \partial_x U_j \psi(x - x_j^t)
	\end{equation}
	and estimating as for $I$ we get
	\[
	\begin{split}
	II &\lesssim_N \fr{1+ \sum_j |\overline{\xi}_j|}{ |(\xi_1 - \xi_2)^2 -
		(\xi_3 - \xi_4)^2|} \int \fr{|D^2 \Phi|}{|D \Phi|^2} \prod 2^{-\ell_j N}
	\chi \Bigl(\fr{x - x_j^t}{2^{\ell_j}} \Bigr) \, \eta dx
	dt\\
	&\lesssim_N 2^{-\ell^* N} \Bigl (\fr{ \langle \xi_1 + \xi_2 - \xi_3 -
		\xi_4\rangle + |\xi_1 - \xi_2 | + |\xi_3 - \xi_4|}{ | (\xi_1 -
		\xi_2)^2 - (\xi_3 - \xi_4)^2|} \Bigr)\fr{\langle \xi_1 + \xi_2 - \xi_3 
		-  \xi_4 \rangle}{ |(\xi_1 - \xi_2)^2 - (\xi_3 - \xi_4)^2|^2}.
	\end{split}
	\]
	%
	%
	
	It remains to consider $III$. The derivatives can distribute in
	various ways:
	\begin{equation}
	\begin{split}\label{e:III_expansion}
	III &\lesssim \fr{1}{ |(\xi_1 - \xi_2)^2 - (\xi_3 - \xi_4)^2|^2} \Bigl \{\int
	|D^2[U_1 \psi(x - x_1^t)] \prod_{j=2}^4 U_j \psi
	(x - x^t_j) \prod_{k=1}^4 \theta_{\ell_k} (x -
	x_k^t)  \eta| \, dx dt\\
	&\quad+  \int \Bigl | D [U_1 \psi(x - x_1^t)] D [U_2
	\psi(x - x_2^t)] \prod_{j=3}^4 U_j \psi(x -
	x^t_j) \prod_{k=1}^4 \theta_{\ell_k} (x -
	x_k^t)  \eta \Bigr | \, dx dt\\
	&\quad+ \int \Bigl| D \prod_j U_j \psi(x -
	x^t_j) D \prod_k \theta_{\ell_k}(x -
	x_k^t) \eta \Bigr | \, dx dt\\
	&\quad+ \int \Bigl| \prod_j U_j \psi(x - x^t_j) D^2
	\prod_{k} \theta_{\ell_k} (x - x_k^t) \eta \Bigr | \, dx
	dt\Bigr\},
	\end{split}
	\end{equation}
	where the first two terms represent sums over the appropriate
	permutations of indices.
	
	We focus on the terms
	involving double derivatives of $U_j$ as the other terms can be
	dealt with as in the estimate for~$II$.  From~\eqref{e:II_one_deriv},
	\begin{equation}
	\label{e:III_two_derivs}
	\begin{split}
	D^2[U_j \psi(x - x_j^t)] &= -i\partial_t V_j (t, x-x_j^t) U_j 
	\psi(x-x_j^t)  
	- (H_j)^2 U_j \psi(x - x_j^t) \\
	& + 2i \overline{\xi^t}_j \partial_x H_j U_j \psi(x - x_j^t) -
	\Bigl[\tfrac{1}{4} \sum_k \partial_x V(t, x_k^t) - \partial_x V(t, x_j^t)
	\Bigr] \partial_x U_j \psi(x - x_j^t) \\
	&+
	(\overline{\xi^t}_j)^2 \partial^2_x U_j \psi(x - x_j^t).
	\end{split}
	\end{equation}
	
	Recalling from~\eqref{e:Vj} that
	\[
	\partial_t V_j(t, x) = x^2 \Bigl [\xi_j^t \int_0^1
	(1-s) \partial_x^3V(t, x_j^t + s x) ds +\int_0^1 (1-s) \partial_t 
	\partial^2_x V
	(t, x_j^t + sx) \, ds \Bigr ],
	\]
	it follows that
	\[
	\begin{split}
	\int \Bigl | \partial_t V_1 (t, x - x_1^t) & U_1 \psi(x - x_1^t)
	\prod_{j=2}^4 U_j \psi(x - x_j^t) \prod_{k=1}^4 \theta_{\ell_{k}} (x
	- x_k^t)\Bigr| \,
	\eta(t) dx dt\\
	&\lesssim 2^{2 \ell_1} \int \Bigl[ \int_0^1 |\xi_1^t \partial^3V_x(t,
	x_1^t + s(x - x_1^t))| \, ds
	\\
	&\quad+ \int_0^1 |\partial_t \partial^2_x V (t, x_j^t + s(x-x_j^t)) \, ds
	\Bigr] \prod_j 2^{-\ell_j N} \chi \Bigl ( \fr{x - x_j^t}{ 2^{\ell_j}
	}
	\Bigr) \, \eta dx dt\\
	&\lesssim_N 2^{-\ell^*N},
	\end{split}
	\]
	where the terms involving $\partial^3_x V$ are handled as in $I_c$ above. 
	Also, 
	from~\eqref{e:wavepacket_schwartz_tails}
	and~\eqref{e:xi_constancy},
	\[
	\begin{split}
	&\int \Bigl | (\overline{\xi}_1^t)^2 \partial_x^2 U_1 \psi(x -
	x_1^t)  \prod_{j=2}^4 U_j \psi(x - x_j^t)
	\prod_{k=1}^4 \theta_{\ell_{k}} (x - x_k^t)\Bigr| \,
	\eta(t) dx dt\lesssim_N \fr{2^{-\ell^*N} (1 + |\overline{\xi}_1|^2) }{1 + |\xi_1 -
		\xi_2| + |\xi_3 - \xi_4|}.
	\end{split}
	\]
	The intermediate terms in~\eqref{e:III_two_derivs} and the other terms
	in the the expansion~\eqref{e:III_expansion} yield similar upper
	bounds. We conclude overall that
	\[
	\begin{split}
	III 
	&\lesssim_N 2^{-\ell^*N} \Bigl( \fr{1}{ |(\xi_1 - \xi_2)^2 -
		(\xi_3 - \xi_4)^2|^2} + \fr{(1 + \sum_j |\overline{\xi}_j|)^2}{ |
		(\xi_1 - \xi_2)^2 - (\xi_3 - \xi_4)^2|^2 \cdot  (1 + |\xi_1 -
		\xi_2| + |\xi_3 - \xi_4|)} \Bigr)\\
	&\lesssim_N 2^{-\ell^*N}\Bigl ( \fr{ 1}{| (\xi_1 - \xi_2)^2 - (\xi_3 -
		\xi_4)^2|^2} + \fr{\langle \xi_1 + \xi_2 - \xi_3 - \xi_4 \rangle^2
		+ (|\xi_1 - \xi_2| + |\xi_3 - \xi_4|)^2}{\bigl| |\xi_1 - \xi_2|^2
		- |\xi_3 - \xi_4|^2 \bigr|^2 \cdot (1+ |\xi_1 - \xi_2| + |\xi_3 - 
		\xi_4|)} \Bigr)\\
	&\lesssim_N 2^{-\ell^*N} \fr{ \langle \xi_1 + \xi_2 - \xi_3 - \xi_4
		\rangle^2 + |\xi_1 - \xi_2| + |\xi_3 - \xi_4|}{ | (\xi_1 - \xi_2)^2
		- (\xi_3 - \xi_4)^2|^2}.
	\end{split}
	\]
	
	Note also that in each of the integrals $I$, $II$, and $III$ we may
	integrate by parts in $x$ to obtain arbitrarily many factors of
	$|\xi_1 + \xi_2 - \xi_3 - \xi_4|^{-1}$. All instances of
	$\langle \xi_1 + \xi_2 - \xi_3 - \xi_4 \rangle$ in the above estimates
	may therefore be replaced by $1$.
	
	Combining $I$, $II$, and $III$, under the 
	hypothesis~\eqref{e:energy_cond} we obtain
	\begin{equation}
	\nonumber
	\begin{split}
	|K^{\vec{\ell}}_0(\vec z)| \lesssim_N 2^{-\ell^*N} \fr{ 1 + |\xi_1 - \xi_2| + 
		|\xi_3 - \xi_4|}{
		\bigl |
		(\xi_1 - \xi_2)^2 - (\xi_3 - \xi_4)^2 \bigr|^2}.
	\end{split}
	\end{equation}
	Combining this with~\eqref{e:x-int_by_parts}, we get
	\begin{equation}
	\begin{split}
	| K^{\vec{\ell}}_0 (\vec z)| &\lesssim_{N_1, N_2} 2^{-\ell^* N_1} \min \Bigl(\frac{\langle
	\xi_1 + \xi_2 - \xi_3 - \xi_4\rangle^{-N_2}}{ 1 + |\xi_1 - \xi_2| + 
		|\xi_3 -
		\xi_4|}, \fr{ 1 + |\xi_1 - \xi_2| + 
		|\xi_3 -
		\xi_4|}{ \bigl | (\xi_1 - \xi_2)^2 - (\xi_3 - \xi_4)^2 \bigr|^2}
	\Bigr)
	\end{split}
	\end{equation}
	for any $N_1, N_2 > 0$. Lemma~\ref{l:K0_ptwise} now follows from 
	summing in
	$\vec{\ell}$.
	

\bibliographystyle{myamsalpha}
\bibliography{../../bibliography}

\end{document}